\documentclass[10pt,a4paper]{article}

\usepackage{graphicx}
\usepackage{epstopdf}
\usepackage{epsfig}
\usepackage{fullpage} 

\usepackage[T1]{fontenc} 
\usepackage[utf8]{inputenc}

\usepackage{amsmath} 
                     
\usepackage{enumitem}
\usepackage{amsfonts}
\usepackage{amssymb}
\usepackage{pstricks}
\usepackage{mathabx}
\author{Benoit}
\title{ratedis}

\newtheorem{theorem}{Theorem}[section]
\newtheorem{lemma}[theorem]{Lemma}
\newtheorem{proposition}[theorem]{Proposition}
\newtheorem{corollary}[theorem]{Corollary}
\newtheorem{remark}[theorem]{Remark}

\newenvironment{proof}[1][Proof]{\begin{trivlist}
\item[\hskip \labelsep {\bfseries #1}]}{\end{trivlist}}
\newenvironment{definition}[1][Definition]{\begin{trivlist}
\item[\hskip \labelsep {\bfseries #1}]}{\end{trivlist}}

\newcommand{\modd}[1]{\vert #1\vert^2}
\newcommand{\nablag}{\nabla^{\gamma}}
\newcommand{\N}{\mathbb{N}}
\newcommand{\psia}{\psi_{(\alpha)}}
\newcommand{\nablamug}{\nabla^{\mu,\gamma}}

\newcommand{\Vd}{\underline{V}}

\newcommand{\ds}{\displaystyle}
\newcommand{\D}{\vert D^\gamma\vert}

\newcommand{\R}{\mathbb{R}}

\newcommand{\B}{\mathfrak{P}}
\newcommand{\rt}{\underline{\mathfrak{a} }}

\newcommand{\dt}{\partial_t}

\def \epsilon {\varepsilon}

\newcommand{\zetaa}{\zeta_{(\alpha)}}

\newcommand{\xig}{\vert\xi^\gamma\vert}

\newcommand{\E}{\mathcal{E}}
\renewcommand{\S}{\mathcal{S}}

\newcommand{\Go}{\mathcal{G}_0}
\newcommand{\G}{\mathcal{G}}

\renewcommand{\S}{\mathcal{S}}
\def \epsilon {\varepsilon}
\begin{document}
\title{The singular limit of the Water-Waves equations in the rigid lid regime}
\author{Mésognon-Gireau Benoît\footnote{UMR 8553 CNRS, Laboratoire de Mathématiques et Applications de l'Ecole Normale Supérieure, 75005 Paris, France. Email: benoit.mesognon-gireau@ens.fr}}
\date{}
\maketitle
\begin{abstract} The re-scaled Water-Waves equations depend strongly on the ratio $\epsilon$ between the amplitude of the wave and the depth of the water. We investigate in this paper the convergence as $\epsilon$ goes to zero of the  free surface Euler equations  to the so called rigid lid model. We first prove that the only solutions of this model are zero. Due to the conservation of the Hamiltonian, the solutions of the free surface Euler equations converge weakly to zero, but not strongly in the general case, as $\epsilon$ goes to zero. We then study this default of convergence. More precisely, we show a strong convergence result of the solutions of the water waves equations in the Zakharov-Craig-Sulem formulation to the solutions of the linear water-waves equations. It is then easy to observe these latter converge weakly to zero. The simple structure of this system also allows us to explain the mechanisms of the weak convergence to zero. Finally, we show that this convergence to the rigid lid model also holds for the solutions of the Euler equations. To this end we give a new proof of the equivalence of the free surface Euler equations and of the Zakharov-Craig-Sulem equation by building an extension of the velocity and pressure fields. \end{abstract}

\section{Introduction}
We recall here classical formulations of the Water Waves problem. We then shortly introduce the meaningful dimensionless parameters of this problem, and introduce the so called rigid lid equations. We then present the main result of this paper. 
\subsection{Formulations of the Water Waves problem}
The Water Waves problem puts the motion of a fluid with a free surface into equations. We recall here two equivalent formulations of the Water Waves equations for an incompressible and irrotationnal fluid. 
\subsubsection{Free surface $d$-dimensional Euler equations} 
The motion, for an incompressible, inviscid and irrotationnal fluid occupying a domain $\Omega_t$ delimited below by a flat bottom and above by a free surface is described by the following quantities:
\begin{itemize}[label=--,itemsep=0pt]
\item the velocity of the fluid $U=(V,w)$, where $V$ and $w$ are respectively the horizontal and vertical components;
 \item the free top surface profile $\zeta$;
 \item the pressure $P.$ 
\end{itemize}
All these functions depend on the time and space variables $t$ and $(X,z) \in\Omega_t$. The domain of the fluid at the time $t$ is given by  $$\Omega_t = \lbrace (X,z)\in\mathbb{R}^{d+1},-H_0 < z <\zeta(t,X)\rbrace,$$ where $H_0$ is the typical depth of the water. The unknowns $(U,\zeta,P)$ are governed by the Euler equations: 
\begin{align}
\begin{cases}
\partial_t U +  U\cdot\nabla_{X,z}U = - \frac{1}{\rho}\nabla P -ge_z\text{ in } \Omega_t\\
\mbox{\rm div}(U) = 0 \text{ in } \Omega_t\\
\mbox{\rm curl}(U) = 0 \text{ in } \Omega_t .
\label{c4:euler}
\end{cases}\end{align}
\par We denote here $-ge_z$ the acceleration of gravity, where $e_z$ is the unit vector in the vertical direction, and $\rho$ the density of the fluid. Here, $\nabla_{X,z}$ denotes the $d+1$ dimensional gradient with respect to both variables $X$ and $z$. \par  \vspace{\baselineskip}
 These equations are completed by boundary conditions: 
\begin{align}
\begin{cases}
\partial_t \zeta +\underline{V}\cdot\nabla\zeta - \underline{w} = 0  \\
U\cdot n = 0 \text{ on } \lbrace z=-H_0\rbrace \\
P=P_{atm}  \text{ on } \lbrace z=\zeta(t,X)\rbrace. \label{c4:boundary_conditions}
\end{cases}
\end{align}
In these equations, $\underline{V}$ and $\underline{w}$ are the horizontal and vertical components of the velocity evaluated at the surface. The vector $n$ in the last equation stands for the normal upward vector at the bottom $(X,z=-H_0)$. We denote $P_{atm}$ the constant pressure of the atmosphere at the surface of the fluid. The first equation of \eqref{c4:boundary_conditions} states the assumption that the fluid particles do not cross the surface, while the last equation of \eqref{c4:boundary_conditions} states the assumption that they do not cross the bottom. The equations \eqref{c4:euler} with boundary conditions \eqref{c4:boundary_conditions} are commonly referred to as the free surface Euler equations.
\subsubsection{Craig-Sulem-Zakharov formulation} 
Since the fluid is by hypothesis irrotational, it derives from a scalar potential: $$U = \nabla_{X,z} \Phi.$$  Zakharov remarked in \cite{zakharov} that the free surface profile $\zeta$ and the potential at the surface $\psi = \Phi_{\vert z=\zeta}$ fully determine the motion of the fluid, and gave an Hamiltonian formulation of the problem. Later, Craig-Sulem, and Sulem (\cite{craigsulem1} and \cite{craigsulem2}) gave a formulation of the Water Waves equation involving the Dirichlet-Neumann operator. The following Hamiltonian system is equivalent (see \cite{david} and \cite{alazard} for more details) to the free surface Euler equations \eqref{c4:euler} and \eqref{c4:boundary_conditions}:
\begin{align}\begin{cases}
        \displaystyle{\partial_t \zeta -  G\psi = 0} \\
        \displaystyle{\partial_t \psi + g\zeta + \frac{1}{2}\vert\nabla\psi\vert^2 - \frac{(G\psi +\nabla\zeta\cdot\nabla\psi)^2}{2(1+\mid\nabla\zeta\mid^2)}=0, }\label{c4:ww_equation} 
\end{cases}\end{align}
where the unknowns are $\zeta$ (free surface profile) and $\psi$ (velocity potential at the surface) with $t$ as time variable and  $X\in\mathbb{R}^d$ as space variable. The fixed bottom profile is $b$, and $G$ stands for the Dirichlet-Neumann operator, that is 

\begin{equation*}
G\psi = G[\zeta]\psi = \sqrt{1+ \modd{\nabla\zeta}} \partial_n \Phi_{\vert z=\zeta},
\end{equation*} 
where $\Phi$ stands for the potential, and solves Laplace equation with Neumann (at the bottom) and Dirichlet (at the surface) boundary conditions:
\begin{align}\begin{cases}
\Delta_{X,z} \Phi = 0 \quad \text{in }  \lbrace (X,z)\in\R^d\times\R, -H_0+ b(X) < z < \zeta(X)\rbrace \\
\phi_{\vert z=\zeta} = \psi,\quad \partial_n \Phi_{\vert z=-H_0} = 0, \label{c4:dirichlet}
\end{cases}\end{align}
with the notation, for the normal derivative $$\partial_n \Phi_{\vert z=-H_0} = \nabla_{X,z}\Phi(X,-H_0)\cdot n$$ where $n$ stands for the normal upward vector at the bottom $(X,-H_0)$. See also \cite{david} for more details.

%

\subsubsection{Dimensionless equations}
Since the properties of the solutions depend strongly on the characteristics of the fluid, it is more convenient to non-dimensionalize the equations by introducing some characteristic lengths of the wave motion:
\begin{enumerate}[itemsep=0pt,label=(\arabic*)]
\item The characteristic water depth $H_0$;
\item The characteristic horizontal scale $L_x$ in the longitudinal direction;
\item The characteristic horizontal scale $L_y$ in the transverse direction (when $d=2$);
\item The size of the free surface amplitude $a_{surf}$;
\end{enumerate}
Let us then introduce the dimensionless variables: $$x'=\frac{x}{L_x},\quad y'=\frac{y}{L_y},\quad \zeta'=\frac{\zeta}{a_{surf}},\quad z'=\frac{z}{H_0},$$ and the dimensionless variables: $$t'=\frac{t}{t_0},\quad \Phi'=\frac{\Phi}{\Phi_0},$$ where $$t_0 = \frac{L_x}{\sqrt{gH_0 }},\quad \Phi_0 = \frac{a_{surf}}{H_0}L_x \sqrt{gH_0}.$$ 

After rescaling, several dimensionless parameters appear in the equation. They are 
\begin{align*}
\frac{a_{surf}}{H_0} = \epsilon, \quad \frac{H_0^2}{L_x^2} = \mu,\quad \frac{L_x}{L_y} = \gamma,
\end{align*}
where  $\epsilon,\mu,\gamma$ are commonly referred to respectively as "nonlinearity", "shallowness" and "transversality" parameters.\par
\vspace{\baselineskip}

For instance, the Zakharov-Craig-Sulem system (\ref{c4:ww_equation}) becomes (see \cite{david} for more details) in dimensionless variables (we omit the "primes" for the sake of clarity): 
\begin{align}\begin{cases}
        \displaystyle{\partial_t \zeta - \frac{1}{\mu} G_{\mu,\gamma}[\epsilon\zeta]\psi = 0 }\\
        \ds \partial_t \psi + \zeta + \frac{\epsilon}{2}\vert\nablag\psi\vert^2 - \frac{\epsilon}{\mu}\frac{(G_{\mu,\gamma}[\epsilon\zeta]\psi +\epsilon\mu\nablag\zeta\cdot\nablag\psi)^2}{2(1+\epsilon^2\mu\mid\nablag\zeta\mid^2)}=0,  \label{c4:ww_equation1}
\end{cases}\end{align} where $G_{\mu,\gamma}[\epsilon\zeta]\psi$ stands for the dimensionless  Dirichlet-Neumann operator, 

\begin{equation*}
G_{\mu,\gamma}[\epsilon\zeta]\psi = \sqrt{1+\epsilon^2 \modd{\nablag\zeta}} \partial_n \Phi_{\vert z=\epsilon\zeta} =  (\partial_z\Phi-\mu\nabla^{\gamma}(\epsilon \zeta)\cdot\nabla^{\gamma}\Phi)_{\vert z=\epsilon\zeta},
\end{equation*} 
where $\Phi$ solves the Laplace equation with Neumann (at the bottom) and Dirichlet (at the surface) boundary conditions 
\begin{equation}\left\{\begin{aligned}
\Delta^{\mu,\gamma} \Phi = 0 \quad \text{in }  \lbrace (X,z)\in\R^d\times\R -1 < z < \epsilon\zeta(X)\rbrace \\
\phi_{\vert z=\epsilon\zeta} = \psi,\quad \partial_n \Phi_{\vert z=-1} = 0.
\end{aligned}\right.\label{c4:dirichletneumannnondim}\end{equation}
We used the following notations:
\begin{align*}
&\nabla^{\gamma} = {}^t(\partial_x,\gamma\partial_y) \quad &\text{ if } d=2\quad &\text{ and } &\nabla^{\gamma} = \partial_x &\quad \text{ if } d=1 \\
&\Delta^{\mu,\gamma} = \mu\partial_x^2+\gamma^2\mu\partial_y^2+\partial_z^2 \quad &\text{ if } d=2\quad &\text{ and } &\Delta^{\mu,\gamma} = \mu\partial_x^2+\partial_z^2 &\quad \text{ if } d=1
\end{align*}
and  $$ \partial_n \Phi_{\vert z=-1}= (\partial_z\Phi-\mu\nabla^{\gamma}(\beta b)\cdot\nabla^{\gamma}\Phi)_{\vert z=-1 }.$$ 

\subsection{The rigid lid model}

 The Euler system \eqref{c4:euler},\eqref{c4:boundary_conditions}, can be written in the dimensionless variables:

\begin{align}\begin{cases}
\displaystyle{\partial_t V +\epsilon (V\cdot\nabla^{\gamma} + \frac{1}{\mu}w\partial_z)V =  -\nabla^{\gamma} P} \text{ in } \Omega_t \\
\displaystyle{\partial_t w + \epsilon (V\cdot\nabla^{\gamma} + \frac{1}{\mu}w\partial_z)w =  -(\partial_z P) }\text{ in } \Omega_t \\
\displaystyle{\partial_t \zeta +\epsilon\underline{V}\cdot\nabla^{\gamma}\zeta - \frac{1}{\mu}\underline{w} = 0 } \\
\nabla^{\mu,\gamma}\cdot U = 0 \text{ in } \Omega_t \\
\mbox{\rm curl}^{\mu,\gamma}(U) = 0 \text{ in } \Omega_t \\
U\cdot n = 0 \text{ for }  z=-1\\
P=\epsilon\zeta \text{ for } z=\epsilon\zeta.\label{c4:eulerystem}
\end{cases}\end{align}
We changed the pressure into hydrodynamic pressure $P=P-(z-\epsilon\zeta)$ and set $P_{atm}=0$.  The rigid lid model models the motion of the fluid as if the top surface was fixed at $z=0$ (see for instance \cite{camassa_holm}). But to derive this model, we do not  brutally take $\epsilon=0$ in the Euler system \eqref{c4:eulerystem} but rather start by a change of time scale, and a change of scale for the pressure:

$$t'=\epsilon t$$ and $$P'=\frac{P}{\epsilon}.$$

Formally, if one takes the limit $\epsilon$ goes to zero in the newly scaled Euler equations, one finds the following so called rigid lid equations (we omit the "primes" for the sake of clarity):
\begin{align}\begin{cases}
\displaystyle{\partial_t V +  (V\cdot\nabla^{\gamma} + \frac{1}{\mu}w\partial_z)V =  -\nabla^{\gamma} P} \text{ in } \Omega \\
\displaystyle{\partial_t w +  (V\cdot\nabla^{\gamma} + \frac{1}{\mu}w\partial_z)w =  -(\partial_z P) }\text{ in } \Omega \\
\underline{w} = 0 \\
\nabla^{\mu,\gamma}\cdot U = 0 \text{ in } \Omega \\
\mbox{\rm curl}^{\mu,\gamma}(U) = 0 \text{ in } \Omega \\
U\cdot n = 0 \text{ for }  z=-1 \\
P=\zeta \text{ for } z=0
\end{cases}\label{c4:rigidlid}\end{align}
where $\Omega$ is now the fixed domain 

$$\Omega = \lbrace (X,z)\in\R^{d+1},-1\leq z\leq 0\rbrace.$$

\subsection{Reminder on the local existence for the Water-Waves equations with flat bottom}
We briefly give some reminders about the local well-posedness theory for the Water-Waves equations and their local existence (see \cite{david} Chapter 4 for a complete study, and also \cite{iguchi2009}). After scaling $$t'=\epsilon t$$ as for the rigid lid equation, one gets from \eqref{c4:ww_equation1} the equations: \begin{align}\begin{cases}
\ds \epsilon\partial_t\zeta-\frac{1}{\mu}G[\epsilon\zeta]\psi = 0 \\
\ds \epsilon\partial_t\psi+\zeta+\frac{\epsilon}{2}\vert\nablag\psi\vert^2-\frac{\epsilon}{\mu}\frac{(G[\epsilon\zeta]\psi+\epsilon\mu\nablag\zeta\cdot\nablag\psi)^2}{2(1+\epsilon^2\mu\vert\nablag\zeta\vert^2}=0.\label{c4:ww_equu}
\end{cases}
\end{align}
As explained above, we prove later that the solutions of the full Water-Waves equations \eqref{c4:ww_equu} converge strongly to the solutions of the linearized equations \eqref{c4:zeta_psi_eq}. The strategy is to treat the non-linear terms of the Water-Waves equations as a perturbation of the linearized equation.\par\vspace{\baselineskip}
Let $t_0>d/2$ and $N\geq t_0+t_0\vee 2 + 3/2$ (where $a\vee b = \sup(a,b)$). The energy for the Water-Waves equations is the following (see Section \ref{c4:notations} for the notations):
$$\mathcal{E}^N(U) = \vert\B\psi\vert_{H^{t_0+3/2}} + \sum_{\vert \alpha\vert\leq N} \vert\zetaa\vert_2 + \vert\psia\vert_2$$
where $\zetaa,\psia$ are the so called Alinhac's good unknowns:
$$\forall \alpha\in\mathbb{N}^d, \zetaa = \partial^\alpha\zeta,\qquad \psia = \partial^\alpha\psi-\epsilon\underline{w}\partial^\alpha\zeta,$$ 
and where $\B$ is the differential operator defined by $$\B = \frac{\D}{(1+\sqrt{\mu}\D)^{1/2}}.$$ The operator $\B$ is of order $1/2$ and acts as the square roots of the Dirichlet-Neumann operator, since there exists $M_0=C(\frac{1}{h_{\min}},\vert\zeta\vert_{H^{t_0+1}})$  where $C$ is a non decreasing function of its arguments, such that \begin{equation}\frac{1}{M_0}\vert\B\vert_2^2\leq (\psi,\frac{1}{\mu}G\psi)_2\leq M_0\vert\B\psi\vert_2^2.\label{c4:equiDN}\end{equation} We consider solutions $U=(\zeta,\psi)$ of the Water-Waves equations in the following space:
\begin{equation*}
E_{T}^N = \lbrace U\in C(\left[ 0,T\right];H^{t_0+2}\times\overset{.}H{}^2(\mathbb{R}^d)), \mathcal{E}^N(U(.))\in L^{\infty}(\left[ 0,T\right])\rbrace.
\end{equation*} The following quantity, called the Rayleigh-Taylor coefficient plays an important role in the Water-Waves problem: \begin{equation*}\rt(\zeta,\psi) = 1+\epsilon(\epsilon\partial_t+\epsilon \underline{V}\cdot\nabla^{\gamma})\underline{w} = -\epsilon\frac{P_0}{\rho a g}(\partial_z P)_{\vert z = \epsilon\zeta}.\label{c4:rtdef}\end{equation*} We can now state the local existence result by Alvarez-Samaniego Lannes (see \cite{alvarez} and \cite{david} Chapter 3):
\begin{theorem}\label{c4:uniform_result}
Let $t_0>d/2$,$N\geq t_0+t_0\vee 2+3/2$. Let $U^0 = (\zeta^0,\psi^0)\in E_0^N$. Let $\epsilon,\gamma$ be such that $$0\leq \epsilon,\gamma\leq 1,$$ and moreover assume that: \begin{equation*}\exists h_{min}>0,\exists a_0>0,\qquad 1+\epsilon\zeta^0\geq h_{min}\quad\text{ and } \quad  \rt(U^0)\geq a_0.\label{c4:rayleigh}\end{equation*} Then, there exists $T>0$ and a unique solution $U^\epsilon\in E_T^N$ to \eqref{c4:ww_equu} with initial data $U^0$. Moreover, $$\frac{1}{T}= C_1,\quad\text{ and }\quad \underset{t\in [0;T]}{\sup} \mathcal{E}^N(U^\epsilon(t)) = C_2$$ with $\ds C_i=C(\mathcal{E}^N(U^0),\frac{1}{h_{min}},\frac{1}{a_0})$ for $i=1,2$.
\end{theorem}
\begin{remark} Note that the time existence provided by Theorem \ref{c4:uniform_result} does not depend on $\epsilon$. Actually, the full Theorem in non flat bottom (see \cite{david}) states that the time of existence for the equations in the original scaling \eqref{c4:ww_equation1} is of size $\frac{1}{\epsilon\vee\beta}$ where $\beta$ is the size of the topography. In our case, $\beta=0$ and therefore one gets a time of existence of size $\frac{1}{\epsilon}$ for \eqref{c4:ww_equation1}, and of size $1$ for the rescaled equations \eqref{c4:ww_equu}. In the case of a non flat bottom, this long time result stands true in presence of surface tension (see \cite{benoit}).
\end{remark}

\subsection{Main result}
We investigate in this paper the rigorous limit $\epsilon$ goes to zero in the rescaled Euler equation \eqref{c4:eulerystem} in view of the mathematical justification of the derivation of the rigid lid model \eqref{c4:rigidlid}. However, as one shall see in Section \ref{c4:solution_rigid}, the only solutions of the rigid lid equations are trivially null. Though the rigid lid model does not then seem of much interest, it implies that the solutions of the rescaled Euler equations \eqref{c4:eulerystem}  converge weakly to zero as $\epsilon$ goes to zero. In \cite{bresch_metivier}, the same limit in the rigid lid regime is investigated by Bresch and Métivier for the Shallow-Water equations with large bathymetry, which consist in an asymptotic model for the Water-Waves problem in the shallow water regime ($\mu$ small). They prove that the rescaled equations are well-posed on a time interval independent on $\epsilon$ (which is equivalent to a large time of existence of size $\frac{1}{\epsilon}$ for the system written in the original variables) and they rigorously pass to the limit as $\epsilon$ goes to zero, and prove the weak convergence of the solutions.\par\vspace{\baselineskip} In \cite{benoit}, a similar long time existence result is proved for the Water-Waves equations with large bathymetry and with surface tension. The local existence result \eqref{c4:uniform_result} implies that (we don't give a precise statement) that there exists $T>0$, and a unique solution $(\zeta,\psi)\in C^1([0;\frac{T}{\epsilon}];H^N\times H^{N+1/2}(\R^d))$ to the Water-Waves equations \eqref{c4:ww_equation1}. Moreover, one has: \begin{equation}
\frac{1}{T}= C_1(\zeta_0,\psi_0),\qquad \vert(\zeta,\psi)\vert_{C^1([0;\frac{T}{\epsilon}];H^N\times H^{N+1/2}(\R^d))} \leq C_2(\zeta_0,\psi_0),\label{c4:ww_boundness}
\end{equation}
where $C_i$ are continuous functions of their arguments, and are independent on $\mu,\epsilon$. The bound on the solutions given by \eqref{c4:ww_boundness} is uniform with respect to $\epsilon$, which allows by a compactness argument to extract a convergent sub-sequence of solutions as $\epsilon$ goes to zero. The limit should be zero (at least for $\psi$) as we discussed above. The question is to understand if the convergence is strong in $C([0;T];H^{N-1}\times H^{N-1/2}(\R^d))$, i.e. "globally in time and space". In this paper, we complete the study of the rigid lid limit for the Water-Waves problem, and give an answer to this question. \par\vspace{\baselineskip} 

Since the equations \eqref{c4:ww_equation1} have the structure of a Hamiltonian equation, the following quantity is conserved: \begin{equation}\frac{1}{2\mu}(\G\psi,\psi)_2+(\zeta,\zeta)_2.\label{c4:hamiltonian}\end{equation} We recall that $\G$ is symmetric and positive for the $L^2$ scalar product. Therefore, even after the time change of scale $$t'=\epsilon t$$ and the limit $\epsilon$ goes to zero, the Hamiltonian is conserved and the strong convergence in $L^2$ of the unknowns cannot be (except in particular cases).  In Section \ref{c4:nonstrongconvergence}, we precisely try to highlight the default of compactness that prevents the strong convergence in dimension $1$, in presence of a flat bottom. As one shall see, the default comes from the linear operator of the Water-Waves equations which has quite a similar behavior to the wave equation. In Section \ref{c4:equivalence}, we prove the equivalence between the Water-Waves equations and the Euler equations, which completes the study of the weak but not strong convergence to zero in Sobolev spaces for the solutions of the Euler equation \eqref{c4:eulerystem} in the rigid lid regime. To sum up, we give here the plan of this article: \begin{itemize}[label=--,itemsep=0pt]
\item In Section \ref{c4:nonstrongconvergence}, we prove that the solutions of the linearized equation does not converge strongly in $L^2(\R^d)$ at a fixed time, in the rigid lid scaling $(t'=\epsilon t)$ for $d=1,2$.
\item In Section \ref{c4:lackstrong}, we prove the strong convergence in $L^\infty([0;T[;L^2\times H^{1/2}(\R^d))$ of the solutions of the full Water-Waves equation in rigid lid scaling to the solutions of the linearized equations, for $d=1$. It proves that the solutions of the full Water-Waves equations do not converge strongly in $L^\infty([0;T[;L^2\times H^{1/2}(\R^d))$ to zero, for $d=1$.
\item In Section \ref{c4:equivalence}, we prove the equivalence between the free-surface Euler equations \eqref{c4:eulerystem} and the Water-Waves equations \eqref{c4:ww_equation1} for $d=1,2$. We show that the solutions of Euler equations converge weakly as $\epsilon$ goes to zero, to the solutions of the rigid lid equation \eqref{c4:rigidlid}. The convergence is therefore not strong, at least for $d=1$, according to the preceding points.
\end{itemize}

\subsection{Notations}\label{c4:notations}
We introduce here all the notations used in this paper.
 \subsubsection{Operators and quantities} Because of the use of dimensionless variables (see before the "dimensionless equations" paragraph), we use the following twisted partial operators: 
\begin{align*}
&\nabla^{\gamma} = {}^t(\partial_x,\gamma\partial_y) \quad &\text{ if } d=2\quad &\text{ and } &\nabla^{\gamma} = \partial_x &\quad \text{ if } d=1 \\
&\Delta^{\mu,\gamma} = \mu\partial_x^2+\gamma^2\mu\partial_y^2+\partial_z^2 \quad &\text{ if } d=2\quad &\text{ and } &\Delta^{\mu,\gamma} = \mu\partial_x^2+\partial_z^2 &\quad \text{ if } d=1 \\
&\nabla^{\mu,\gamma} = {}^t(\sqrt{\mu}\partial_x,\gamma\sqrt{\mu}\partial_y,\partial_z)\quad &\text{ if } d=2\quad &\text{ and } &{}^t(\sqrt{\mu}\partial_x,\partial_z) &\quad \text{ if } d=1 \\
&\nabla^{\mu,\gamma}\cdot = \sqrt{\mu}\partial_x+\gamma\sqrt{\mu}\partial_y+\partial_z\quad &\text{ if } d=2\quad &\text{ and } &\sqrt{\mu}\partial_x+\partial_z &\quad \text{ if } d=1 \\
&\mbox{\rm curl}^{\mu,\gamma} = {}^t(\sqrt{\mu}\gamma\partial_y-\partial_z,\partial_z-\sqrt{\mu}\partial_x,\partial_x-\gamma\partial_y)&\text {if } d=2.\quad
\end{align*}
\begin{remark}All the results proved in this paper do not need the assumption that the typical wave lengths are the same in both directions, ie $\gamma = 1$. However, if one is not interested in the dependence of $\gamma$, it is possible to take $\gamma = 1$ in all the following proofs. A typical situation where $\gamma\neq 1$ is for weakly transverse waves for which $\gamma=\sqrt{\mu}$; this leads to weakly transverse Boussinesq systems and the Kadomtsev–Petviashvili equation (see \cite{lannes_saut}).
\end{remark}
For all $\alpha = (\alpha_1,..,\alpha_d)\in\mathbb{N}^d$, we write $$\partial^\alpha = \partial^{\alpha_1}_{x_1}...\partial^{\alpha_d}_{x_d}$$ and $$\vert\alpha\vert = \alpha_1+...+\alpha_d.$$

 We use the classical Fourier multiplier 
$$\Lambda^s = (1-\Delta)^{s/2} \text{ on } \mathbb{R}^d$$ defined by its Fourier transform as $$\mathcal{F}(\Lambda^s u)(\xi) = (1+\vert\xi\vert^2)^{s/2}(\mathcal{F}u)(\xi)$$ for all $u\in\mathcal{S}'(\mathbb{R}^d)$.
The operator $\mathfrak{P}$ is defined as 
\begin{equation}
\mathfrak{P} = \frac{\vert D^{\gamma}\vert}{(1+\sqrt{\mu}\vert D^{\gamma}\vert)^{1/2}}\label{c4:defp}
\end{equation}
where $$\mathcal{F}(f(D)u)(\xi) = f(\xi)\mathcal{F}(u)(\xi)$$ is defined for any smooth function $f$ of polynomial growth and $u\in\mathcal{S}'(\mathbb{R}^d)$. The pseudo-differential operator $\mathfrak{P}$ acts as the square root of the Dirichlet-Neumann operator (see \eqref{c4:equiDN}).\\	
We denote as before by $G_{\mu,\gamma}$ the Dirichlet-Neumann operator, which is defined as followed in the scaled variables:

\begin{equation*}
G_{\mu,\gamma}\psi = G_{\mu,\gamma}[\epsilon\zeta]\psi = \sqrt{1+\epsilon^2 \modd{\nablag\zeta}} \partial_n \Phi_{\vert z=\epsilon\zeta} =  (\partial_z\Phi-\mu\nabla^{\gamma}(\epsilon\zeta)\cdot\nabla^{\gamma}\Phi)_{\vert z=\epsilon\zeta},
\end{equation*} 
where $\Phi$ solves the Laplace equation 
\begin{align*}
\begin{cases}
\Delta^{\gamma,\mu}\Phi = 0\\
\Phi_{\vert z=\epsilon\zeta} = \psi,\quad \partial_n \Phi_{\vert z=-1} = 0.
\end{cases}
\end{align*}

For the sake of simplicity, we use the notation $G[\epsilon\zeta]\psi$ or even $G\psi$ when no ambiguity is possible.

\subsubsection{The Dirichlet-Neumann problem}
In order to study the Dirichlet-Neumann problem \eqref{c4:dirichlet}, we need to map $\Omega_t$ into a fixed domain (and not on a moving subset). For this purpose, we introduce the following fixed strip:
$$\mathcal{S} = \mathbb{R}^d\times (-1;0)$$
and the diffeomorphism 
\begin{equation}
\Sigma_t^{\epsilon} :
\left.
  \begin{aligned}
    \mathcal{S} & \rightarrow &\Omega_t \\
    (X,z) & \mapsto & (1+\epsilon\zeta(X))z+\epsilon\zeta(X). \\
  \end{aligned}
\right.\label{c4:diffeoS}
\end{equation}
It is quite easy to check that $\Phi$ is the variational solution of \eqref{c4:dirichlet} if and only if $\phi = \Phi\circ\Sigma_t^{\epsilon}$ is the variational solution of the following problem:
\begin{align}\begin{cases}
        \nabla^{\mu,\gamma}\cdot P(\Sigma_t^{\epsilon})\nabla^{\mu,\gamma} \phi = 0  \label{c4:dirichletneumann}\\
        \phi_{z=0}=\psi,\quad \partial_n\phi_{z=-1} = 0,  \end{cases}
\end{align}
and where $$P(\Sigma_t^{\epsilon}) = \vert \det  J_{\Sigma_t^{\epsilon}}\vert J_{\Sigma_t^{\epsilon}}^{-1}~^t(J_{\Sigma_t^{\epsilon}}^{-1}),$$ where $J_{\Sigma_t^{\epsilon}}$ is the jacobian matrix of the diffeomorphism $\Sigma_t^{\epsilon}$. For a complete statement of the result, and a proof of existence and uniqueness of solutions to these problems, see \cite{david} Chapter 2. \par 
We introduce here the notations for the shape derivatives of the Dirichlet-Neumann operator. More precisely, we define the  open set  $\mathbf{\Gamma}\subset H^{t_0+1}(\mathbb{R}^d)$ as:
$$\mathbf{\Gamma} =\lbrace \Gamma=\zeta\in H^{t_0+1}(\mathbb{R}^d),\quad \exists h_0>0,\forall X\in\mathbb{R}^d, \epsilon\zeta(X) +1- \geq h_0\rbrace$$ and, given a $\psi\in \overset{.}H{}^{s+1/2}(\mathbb{R}^d)$, the mapping: \begin{displaymath}G[\epsilon\cdot] : \left. \begin{array}{rcl}
&\mathbf{\Gamma} &\longrightarrow H^{s-1/2}(\mathbb{R}^d) \\
&\Gamma=(\zeta,b) &\longmapsto G[\epsilon\zeta]\psi.
\end{array}\right.\end{displaymath} We can prove the differentiability of this mapping. See Appendix \ref{c4:appendixA} for more details. We denote $d^jG(h,k)\psi$ the $j$-th derivative of the mapping at $(\zeta,b)$ in the direction $(h,k)$. When we only differentiate in one direction, and no ambiguity is possible, we simply denote $d^jG(h)\psi$ or $d^j G(k)\psi$.

\subsubsection{Functional spaces}
The standard scalar product on $L^2(\mathbb{R}^d)$ is denoted by $(\quad,\quad)_2$ and the associated norm $\vert\cdot\vert_2$. We will denote the norm of the Sobolev spaces $H^s(\mathbb{R}^d)$ by $\vert \cdot\vert_{H^s}$.\par

We introduce the following functional Sobolev-type spaces, or Beppo-Levi spaces: 
\begin{definition}
We denote $\dot{H}^{s+1}(\mathbb{R}^d)$ the topological vector space 
$$\dot{H}^{s+1}(\mathbb{R}^d) = \lbrace u\in L^2_{loc}(\mathbb{R}^d),\quad \nabla u\in H^s(\mathbb{R}^d)\rbrace$$
endowed with the (semi) norm $\vert u\vert_{\dot{H}^{s+1}(\mathbb{R}^d)} = \vert\nabla u\vert_{H^s(\mathbb{R}^d)} $.
\end{definition}
Just remark that $\dot{H}^{s+1}(\mathbb{R}^d)/\mathbb{R}^d$ is a Banach space (see for instance \cite{lions}).\par

The space variables $z\in\mathbb{R}$ and $X\in\mathbb{R}^d$ play different roles in the equations since the Euler formulation (\ref{c4:euler}) is posed for $(X,z)\in \Omega_t$. Therefore, $X$ lives in the whole space $\mathbb{R}^d$ (which allows to take fractional Sobolev type norms in space), while $z$ is actually bounded.  For this reason, we need to introduce the following Banach spaces: 
\begin{definition}
The Banach space $(H^{s,k}((-1,0) \times\mathbb{R}^d),\vert .\vert_{H^{s,k}})$ is defined by 
$$H^{s,k}((-1,0) \times\mathbb{R}^d) = \bigcap_{j=0}^{k} H^j((-1,0);H^{s-j}(\mathbb{R}^d)),\quad \vert u\vert_{H^{s,k}} = \sum_{j=0}^k \vert \Lambda^{s-j}\partial_z^j u\vert_2.$$
\end{definition}

\section{The rigid lid limit for the Water-Waves equations}\label{c4:rigidlidlimitww}
We prove in this section that the solutions of the linearized Water-Waves equations in flat bottom converge weakly but non strongly in $L^2$ as $\epsilon$ goes to zero, in the rigid lid regime. We then prove the strong convergence in $L_t^\infty L_x^2$ of the solutions of the full Water-Waves equations to the solutions of the linearized equations, in the same regime, in dimension $1$. 
\begin{remark} The hypothesis "$d=1$" can be removed, if one can prove a dispersive estimate for the linear operator of the Water-Waves equation of the form of Theorem \ref{c4:dispersive_estimate}, in dimension $d=2$. The  flat bottom hypothesis seems however less easy to remove, due to technical reasons, such as the complexity of the asymptotic expansion of $\G$ with respect to the surface when the bottom is non flat.
\end{remark}

\subsection{Solutions of the rigid lid equations}\label{c4:solution_rigid} We recall the formulation of the rigid lid equations with a flat bottom (see also \cite{camassa_holm} for reference):
\begin{align}\begin{cases}
\displaystyle{\partial_t V +  (V\cdot\nabla^{\gamma} + \frac{1}{\mu}w\partial_z)V =  -\nabla^{\gamma} P} \text{ in } \Omega\\
\displaystyle{\partial_t w +  (V\cdot\nabla^{\gamma} + \frac{1}{\mu}w\partial_z)w =  -(\partial_z P) }\text{ in } \Omega \\
\underline{w} = 0 \\
\nabla^{\mu,\gamma}\cdot U = 0 \text{ in } \Omega \\
\mbox{\rm curl}^{\mu,\gamma}(U) = 0 \text{ in } \Omega \\
U\cdot e_z = 0 \text{ for }  z=-1 \\
P=\zeta \text{ for } z=0
\end{cases}\label{c4:rigidlid1}\end{align}
where $\Omega$ is the fixed domain 
$$\Omega = \lbrace (X,z)\in\R^{d+1},-1\leq z\leq 0\rbrace.$$
We prove here that in fact, there is only one trivial solution to the system \eqref{c4:rigidlid1}:

\begin{lemma}\label{c4:div_curl}
Let us consider the following div-curl problem:
\begin{align}\label{c4:div_curl_eq}
\begin{cases}
\nabla^{\mu,\gamma}\cdot U = 0 \text{ in } \Omega \\
\mbox{\rm curl}^{\mu,\gamma}(U) = 0 \text{ in } \Omega\\
U\cdot e_z = 0 \text{ for }  z=-1 \\
\underline{w} = 0
\end{cases}
\end{align}
with the above notations. Then, the unique solution in $\dot{H}^1(\Omega)$ to the system \eqref{c4:div_curl_eq} is $U=0$.
\end{lemma}
\begin{proof}
The vector field $U=0$ is of course a solution to the system \eqref{c4:div_curl_eq}. Let us now consider a solution $U$ to this problem. The curl free condition over $U$ states that there exists a potential $\Phi$ such that $U=\nabla^{\mu,\gamma} \Phi$ in $\Omega$. The divergence free condition provides $\Delta^{\mu,\gamma} \Phi = 0$ in $\Omega$. With the boundary conditions (recall that $\underline{w}$ is the vertical component of the velocity at the surface), the potential $\Phi$ satisfies the following Laplace equation:
\begin{align*}
\begin{cases}
\Delta^{\mu,\gamma}\Phi = 0 \qquad \text{ in } \Omega \\
\nabla^{\mu,\gamma}\Phi_{\vert z=0}\cdot n = 0 \\
 \nabla^{\mu,\gamma} \Phi_{\vert z=-1}\cdot n = 0,
\end{cases}
\end{align*}
where $n$ denotes the upward normal vector in the vertical direction. It is well known (see for instance \cite{david} Chapter 2 and Appendix A) that the unique solution to this system in $\dot{H}^1(\Omega)$ satisfies $\nabla^{\mu,\gamma} \Phi = 0$.\footnote{In fact the result stands even if the bottom is non flat.}
%
%
%
%
$\Box$

\end{proof}
Let us explain this result on a physical point of view. The rigid-lid equations are obtained after scaling the time $$t'=\epsilon t$$ in the free surface Euler equations, and passing to the limit as $\epsilon$ goes to zero. Therefore, in the original time variable, $t=\frac{t'}{\epsilon}$, the rigid lid limit consists in letting the amplitude of the waves goes to zero AND moving forward in time at a rate $\frac{1}{\epsilon}$. At the limit, after an infinite time, all the interesting components of the waves moved to $\pm\infty$ and there only remains a static fluid ($U=0$) with a flat surface. It is possible that some vortices remain, but they are not seen by the model ($\rm curl(U)=0$). One could generalize the study lead in this paper to the Water-Waves equation with vorticity (see \cite{castro_lannes1}).

\subsection{The linearized equation around $\zeta=0$ in the rigid lid regime}\label{c4:nonstrongconvergence}
We prove here that the solutions to the Water-Waves equations do not converge strongly in $L^2$ as $\epsilon$ goes to zero. To this purpose, we study the lack of compactness induced by the linear operator of the Water-Waves equations. We recall that in the rigid lid regime, we perform a time scaling $$t'=\epsilon t.$$ The Water-Waves equations can be written in dimensionless form, and in the rigid lid scaling:
\begin{align}\begin{cases}
\ds \epsilon\partial_t\zeta-\frac{1}{\mu}G[\epsilon\zeta]\psi = 0\\
\ds \epsilon\partial_t\psi+\zeta+\frac{\epsilon}{2}\vert\nablag\psi\vert^2-\frac{\epsilon}{\mu}\frac{(G[\epsilon\zeta]\psi+\epsilon\mu\nablag\zeta\cdot\nablag\psi)^2}{2(1+\epsilon^2\mu\vert\nablag\zeta\vert^2}=0.\label{c4:water_waves}
\end{cases}
\end{align}
Since we are interested in the convergence as $\epsilon$ goes to zero, we write this system under the form:
\begin{equation}\left\{\begin{aligned}
\epsilon\partial_t\zeta-\frac{1}{\mu}\Go\psi = \epsilon f\\
\epsilon\partial_t\psi+\zeta = \epsilon g.\\
\end{aligned}\right.\label{c4:splitt1}\end{equation}
We will treat the non-linear terms $f,g$ present in \eqref{c4:splitt1} as a perturbation of the linearized equations. To this purpose,  we start to study the solutions of the linearized system:
\begin{align}\begin{cases}
\epsilon\partial_t\zeta-\frac{1}{\mu}\Go\psi=0 \\
\epsilon\partial_t\psi+\zeta = 0, \label{c4:zetapsi_eq}
\end{cases}\end{align}
where $\Go = G[0]$ is the Dirichlet-Neumann operator in flat bottom and flat surface, and is given by (see for instance \cite{david} Chapter 1):
\begin{equation}\Go\psi = \sqrt{\mu}\D\tanh(\sqrt{\mu}\D)\psi\label{c4:defgo}\end{equation} for all $\psi\in\mathcal{S}(\R^d)$.
It is easy to check that the solution $\zeta$ of \eqref{c4:zetapsi_eq} satisfies the following equation:
\begin{equation}
\epsilon^2\partial_t\zeta+\frac{1}{\mu}\Go\zeta = 0.\label{c4:zeta_eq}\end{equation}
The equation \eqref{c4:zeta_eq} looks like a wave equation, except that $\Go$ is not exactly an order two operator (since it acts like an order one operator for high frequencies). In the case of the wave equation on $\R$ \begin{align*}\begin{cases}\epsilon^2\partial_t u -\partial_x^2 u = 0 \\
u(x,0) = u_0 \\
(\epsilon\partial_t u)(x,0) = u_1,
\end{cases}\end{align*} with $u_0$, $u_1$ smooth, the solutions have components of the form $$u_0(x-\frac{1}{\epsilon}t)+u_0(x+\frac{1}{\epsilon}t)$$ (plus other terms we do not detail). Each of these two components converges pointwise to $0$ as $\epsilon$ goes to zero (since $u_0(t,x) \underset{\vert x\vert \rightarrow +\infty}{\longrightarrow} 0$) and does not converge strongly in $L^2(\R)$, since for instance its $L^2(\R)$ norm  does not depend on $\epsilon$. The same behavior, combined with some dispersive effects stands for \eqref{c4:zeta_eq}. We start to give an existence result for the linearized equations \eqref{c4:zetapsi_eq}.  Note that it is not difficult to prove formally that if $(\zeta,\psi)$ is the solution of \eqref{c4:zetapsi_eq}, then the following quantity, called "Hamiltonian" is conserved through time:
$$\frac{1}{2\mu}(\Go\psi,\psi)_2+\frac{1}{2}\vert\zeta\vert_2^2.$$ Therefore, we introduce the following operator: $$\B = \frac{\D}{(1+\sqrt{\mu}\D)^{1/2}}.$$ Note that $\B$ has the same behavior as the square root of $\Go$ defined by $\eqref{c4:defgo}$.

\begin{proposition}\label{c4:existence_linear}	
Let $s\geq 1$. Let also $\zeta_0\in H^s(\R^d)$ and $\psi_0$ be such that $\B\psi_0\in H^s(\R^d)$. Then, there exists a unique solution $(\zeta,\psi)$ to the equation:
\begin{align}\begin{cases}
\epsilon\partial_t\zeta-\frac{1}{\mu}\Go\psi=0 \\
\epsilon\partial_t\psi+\zeta = 0 \\
(\zeta(0,X),\psi(0,X)) = (\zeta^0(X),\psi^0(X))  \label{c4:zeta_psi_eq}
\end{cases}\end{align}
such that $(\zeta,\B\psi)\in C(\R,H^s(\R^d))\cap C^1(\R,H^{s-1}(\R^d))$. Moreover, one has: $$\forall t\in\R,\qquad \begin{pmatrix}\zeta \\ \psi\end{pmatrix}(t) = e^{-\frac{t}{\epsilon}L} \begin{pmatrix}\zeta_0 \\\psi_0\end{pmatrix}$$ where \begin{equation}e^{tL} = \frac{1}{2}\begin{pmatrix}
1 &-i\omega(D) \\
-\frac{1}{i\omega(D)} &1
\end{pmatrix}e^{-i\omega(D)t}+\frac{1}{2}\begin{pmatrix}
1 &i\omega(D) \\
\frac{1}{i\omega(D)} &1 \end{pmatrix}e^{i\omega(D)t}\label{c4:def_L}\end{equation} and \begin{equation}\omega(D) = \sqrt{\frac{\D\tanh(\sqrt{\mu}\D)}{\sqrt{\mu}}}.\label{c4:def_omega}\end{equation}
\end{proposition}
\begin{remark} The notation $e^{tL}$ refers to the linear operator of the Water-Waves, defined by $$L= \begin{pmatrix}
0 &-\frac{1}{\mu}\Go \\
1 &0
\end{pmatrix}.$$
\end{remark}
\begin{proof}

Let us fix $s\geq 1$, and $(\zeta_0,\psi_0) \in H^s(\R^d)^2$. If $(\zeta,\psi)$ is a solution to \eqref{c4:zeta_psi_eq}, then $\zeta$ satisfies the following equation:
\begin{equation}\left\{\begin{aligned}
\epsilon^2\partial_t^2\zeta+\frac{1}{\mu}\Go\zeta = 0 \\
\zeta(x,0) = \zeta_0\\
(\epsilon\partial_t \zeta)(x,0) = \zeta_1,
\end{aligned}\right.\label{c4:equilol}\end{equation}
in the space $C(\R,H^s(\R^d))\cap C^1(\R,H^{s-1}(\R^d))$, with $\zeta_1 = \frac{1}{\mu}\Go\psi_0$. As usual, we take the Fourier transform in space of the equation \eqref{c4:equilol} and we denote $\widehat{\zeta}(t,\xi)$ the Fourier transform of $\zeta$ with respect to the variable $X$ for a fixed $t$. The distribution $\widehat{\partial_t \zeta} $of $\mathcal{S}'(\R\times\R^d)$ is equal to the distribution $\partial_t\widehat{\zeta}$ using the Banach-Steinhaus theorem. We therefore get the following equation in the distributional sense of $\mathcal{S}'(\R\times\R^d)$:
\begin{align*}
\begin{cases}
\epsilon^2\partial_t^2 \widehat{\zeta}+\frac{1}{\sqrt{\mu}}\xig\tanh(\sqrt{\mu}\xig)\widehat{\zeta}=0\\
\widehat{\zeta}(\xi,0) = \widehat{\zeta_0}(\xi) \\
(\epsilon\partial_t \widehat{\zeta})(\xi,0) = \widehat{\zeta_1}(\xi).
\end{cases}
\end{align*}
It is then easy to check that the unique solution to this equation is given by

\begin{equation}\widehat{\zeta}(\xi,t) = (\frac{\widehat{\zeta_0}(\xi)-i\omega(\xi) \widehat{\psi_0}(\xi)}{2})\exp^{i\omega(\xi)\frac{t}{\epsilon}}+(\frac{\widehat{\zeta_0}(\xi)+i\omega(\xi) \widehat{\psi_0}(\xi)}{2})\exp^{-i\omega(\xi)\frac{t}{\epsilon}}\label{c4:sol_zeta}\end{equation}
with $$\omega(\xi) = \sqrt{\frac{\xig\tanh(\sqrt{\mu}\xig)}{\sqrt{\mu}}}
$$ such that $\frac{1}{\mu}\Go = \omega(D)^2$. The desired regularity is easy to get by noticing that \begin{align*}\omega(\xi)\widehat{\psi_0}(\xi)& = \frac{\xig}{(1+\sqrt{\mu}\xig)^{1/2}}\widehat{\psi_0}\sqrt{\frac{\tanh(\sqrt{\mu}\xig)(1+\sqrt{\mu}\xig)}{\sqrt{\mu}\xig}} \\&=\widehat{\B\psi_0}\sqrt{\frac{\tanh(\sqrt{\mu}\xig)(1+\sqrt{\mu}\xig)}{\sqrt{\mu}\xig}}. \end{align*} with $\B\psi_0\in H^s(\R^d)$ and that there exists $C_1$, $C_2>0$ independent of $\mu,\epsilon$ such that:
\begin{equation*}\forall \xi\in\R^d,\qquad C_1\leq \sqrt{\frac{\tanh(\sqrt{\mu}\xig)(1+\sqrt{\mu}\xig)}{\sqrt{\mu}\xig}} \leq C_2.\end{equation*} The same method applies for $\psi$ and one gets that if $(\zeta,\psi)$ satisfies \eqref{c4:zeta_psi_eq}, then necessarily: 
\begin{equation}\widehat{\psi}(t,\xi) = \widehat{\psi_0}\cos(\omega(\xi)\frac{t}{\epsilon})-\frac{\sin(\omega(\xi)\frac{t}{\epsilon})}{\omega(\xi)}\widehat{\zeta_0}.\label{c4:sol_psi}\end{equation}
It is easy to check that $(\zeta,\psi)$ given by \eqref{c4:sol_zeta}, \eqref{c4:sol_psi} is a solution to the system \eqref{c4:zetapsi_eq}. 
  $\Box$
\end{proof}

The oscillating component of the solution of the linearized Water-Waves equations \eqref{c4:zeta_psi_eq} appears in the explicit formula \eqref{c4:sol_zeta}, but we need to prove an oscillating phase method type result in order to conclude to the weak convergence to zero and strong non-convergence. The phase $\omega$ is actually not differentiable at
 $\xi=0$, which prevent the direct use of standards results on stationary phase methods.

\begin{proposition}\label{c4:oscillatory}Let $t>0$ and $u\in C_0^1(\R^d)$. Then, one has $$\int_{\R^d} \exp^{i\frac{t}{\epsilon}\omega(\xi)} u(\xi)d\xi\underset{\epsilon\rightarrow 0}{\longrightarrow} 0.$$ 
\end{proposition}\

\begin{proof} One can compute:

\begin{equation}
\forall \xi\neq 0,\quad (\nabla \omega)(\xi) = \frac{\tanh(\xig)+\xig(1-\tanh^2(\xig))}{2\sqrt{\xig\tanh(\xig)}}\frac{\xi}{\xig}.\label{c4:derivee_omega}\end{equation}
It is easy to see that $\nabla\omega$ does not vanish for $\xi\neq 0$. It is also easy to show that $\omega$ is not differentiable in $\xi=0$. But the derivative $\nabla\omega$ stay bounded as $\xi$ goes to zero.  We treat separately the cases $d=1$ and $d>1$. \par

\textit{- Case $d=1$}
One writes
\begin{equation}
\int_\R \exp^{i\frac{t}{\epsilon}\omega(\xi)} u(\xi)d\xi = \int_{-\infty}^0 \exp^{i\frac{t}{\epsilon}\omega(\xi)} u(\xi)d\xi + \int_0^{+\infty} \exp^{i\frac{t}{\epsilon}\omega(\xi)} u(\xi)d\xi.\label{c4:split_integral}  \end{equation}
Since $\omega$ is right and left differentiable in $\xi=0$, according to the expression \eqref{c4:derivee_omega}, one can  use standard methods of stationary phase. Indeed, one can write:

\begin{align*} \int_{-\infty}^0 \exp^{i\frac{t}{\epsilon}\omega(\xi)} u(\xi)d\xi &= \frac{\epsilon}{ti}\int_{-\infty}^0 \frac{1}{\omega'(\xi)}\frac{d}{d\xi}(\exp^{i\frac{t}{\epsilon}\omega(\xi)})u(\xi)d\xi \\
&= \frac{\epsilon}{ti}\frac{u(0)}{\omega'^-(0)}e^{i\frac{t}{\epsilon}\omega(0)}-\frac{\epsilon}{ti} \int_{-\infty}^0 e^{i\frac{t}{\epsilon}\omega(\xi)}\frac{d}{d\xi}(\frac{u(\xi)}{\omega'^{-}(\xi)})d\xi
\end{align*}
where $\omega'^{+}(0)$ and $\omega'^{-}(0)$ are respectively the right and left derivatives of $\omega$ at $\xi=0$. We then get the desired result. The same goes for the second integral of \eqref{c4:split_integral}.\par

\textit{- Case $d>1$}
We use the standard integrating by part method for the truncated integral:
\begin{align}
\int_{\vert\xi\vert\geq\delta}e^{i\frac{t}{\epsilon}\omega(\xi)} u(\xi)d\xi &=\int_{\vert\xi\vert\geq\delta} \frac{\epsilon}{it}\frac{1}{\vert \omega'(\xi)\vert^2} \sum_{j=1}^d \frac{\partial}{\partial\xi_j}\omega (\xi)\frac{\partial}{\partial\xi_j}( e^{i\frac{t} {\epsilon}\omega(\xi)})u(\xi)d\xi \nonumber\\
&= \frac{\epsilon}{it} \int_{\vert\xi\vert = \delta}  \frac{1}{\vert \omega'(\xi)\vert^2} \sum_{j=1}^d \frac{\partial}{\partial\xi_j}\omega (\xi) e^{i\frac{t} {\epsilon}\omega(\xi)}u(\xi)n_j d\sigma(\xi)\nonumber\\
&-\frac{\epsilon}{it}\int_{\vert\xi\vert\geq\delta} \sum_{j=1}^d\frac{\partial}{\partial\xi_j}(\frac{1}{\vert \omega'(\xi)\vert^2}  \frac{\partial}{\partial\xi_j}\omega (\xi)u(\xi)) e^{i\frac{t} {\epsilon}\omega(\xi)}d\xi\label{c4:phase_stat}
\end{align}
where $n_j$ stands for the $j$-th component of the normal external vector at the surface $\vert\xi\vert = \delta$. Now, one can check that:

$$\frac{1}{\vert\omega'(\xi)\vert^2}\frac{\partial}{\partial\xi_j} \omega(\xi) = \frac{2\xi_j\sqrt{\xig\tanh(\xig)}}{\xig\tanh(\xig)+(1-\tanh^2(\xig))\xig^2}$$
and thus this function is bounded as $\xi$ goes to zero. Moreover, one can check by computation that $$\vert \frac{\partial}{\partial\xi_j} (\frac{1}{\vert\omega'(\xi)\vert^2}\frac{\partial}{\partial\xi_j} \omega(\xi))\vert \leq \frac{C}{\vert\xi\vert}$$ as $\xi$ goes to zero, and consequently this function is integrable at $\xi=0$ (remember that $d>1$ here). Therefore, one can pass the limit as $\delta$ goes to zero in the formula \eqref{c4:phase_stat} and get 

$$\int_{\R^d}e^{i\frac{t}{\epsilon}\omega(\xi)} u(\xi)d\xi =  -\frac{\epsilon}{it}\int_{\R^d} \sum_{j=1}^d\frac{\partial}{\partial\xi_j}(\frac{1}{\vert \omega'(\xi)\vert^2}  \frac{\partial}{\partial\xi_j}\omega (\xi)u(\xi)) e^{i\frac{t} {\epsilon}\omega(\xi)}d\xi$$
and we get the desired result. 
\end{proof} \qquad$\Box$

Now it is easy to prove the following result:

\begin{theorem}
Let $s\geq 1$. Let $(\zeta_0,\psi_0)$ be such that $(\zeta_0,\B\psi_0)\in H^s(\R^d)$.  Then, for all $t\in\R$, the solution $(\zeta,\psi)(t)$ of the system \eqref{c4:zetapsi_eq} converges weakly to zero, and does not converge strongly, in $L^2(\R^d)$. \label{c4:theorem_nonstrongcv}
\end{theorem}
\begin{proof}
We prove it for $\zeta$. We first assume that $\widehat{\zeta},\widehat{\psi}\in C_0^1(\R^d)$. The weak convergence to zero is a consequence of the explicit formulation of the solution given by \eqref{c4:sol_zeta} and the Proposition \ref{c4:oscillatory}. To prove the strong non convergence, one computes:
\begin{align*}
\vert\zeta(t,.)\vert_2^2 &= \frac{1}{(2\pi)^d}\vert \widehat{\zeta}(t,.)\vert_2^2 \\
&=\frac{1}{(2\pi)^d}\int_{\R^d} \frac{1}{2} \big( (1+\cos(2\omega(\xi)\frac{t}{\epsilon}))\vert\widehat{\zeta_0}(\xi)\vert^2 + &&\omega(\xi)^2(1-\cos(2\omega(\xi)\frac{t}{\epsilon}))\vert\widehat{\psi_0}(\xi)\vert^2\big)\\& &&+Re(\widehat{\zeta_0}\overline{\widehat{\psi_0}})(\xi)\sin(2\omega(\xi)\frac{t}{\epsilon})d\xi
\end{align*} 
where we expanded the square modulus of the explicit formulation of $\widehat{\zeta}$ given by \eqref{c4:sol_zeta}. We can then conclude, using Proposition \ref{c4:oscillatory} to the following convergence:
$$\vert\zeta\vert_2^2 \underset{\epsilon\rightarrow 0}{\longrightarrow} \frac{1}{2} (\vert \zeta_0\vert_2^2+\vert \omega(D)\psi_0\vert_2^2)$$ and thus $\zeta(t,.)$ does not converge strongly to zero. \par\vspace{\baselineskip}For the general case, one proceeds by density of $C_0^1(\R^d)$ in $L^2(\R^d)$. \qquad$\Box$
\end{proof}
\begin{remark}As explained above, if $\zeta,\psi$ is a solution to \eqref{c4:zeta_psi_eq}, then the Hamiltonian $$\frac{1}{2\mu}(\Go\psi,\psi)_2+\frac{1}{2}\vert\zeta\vert_2^2$$ is conserved through time. Since $\frac{1}{\mu}(\Go\psi,\psi)_2 \sim \vert\B\psi\vert_2^2$, the strong convergence of both $\zeta$ and $\psi$ to zero cannot happen (except for zero initial conditions). However, it would have been possible that some transfers of energy occur between $\zeta$ and $\psi$, with one of the unknown strongly converging to zero. The Theorem \ref{c4:theorem_nonstrongcv} states that such behavior does not happen. \end{remark}

\subsection{Lack of strong convergence for the full Water-Waves equations in dimension 1}\label{c4:lackstrong}
In all this section, we work with a flat bottom and in dimension $1$: $$d=1,\qquad b=0.$$ The study in dimension $2$ should however not be hard to do, if one gets a dispersive estimate of the form of Theorem \ref{c4:dispersive_estimate} for the Water-Waves operator in dimension $2$. In the previous section, we proved that the solutions of the linearized Water-Waves equations \eqref{c4:zetapsi_eq} weakly converge as $\epsilon$ goes to zero in $L^2(\R)$, and do not converge strongly. We now establish a similar result for the solutions of the full nonlinear Water-Waves system \eqref{c4:ww_equu}. To this purpose, we write this system under the form:
\begin{align*}
\begin{cases}
\epsilon\partial_t\zeta-\frac{1}{\mu}\Go\psi = \epsilon f\\
\epsilon\partial_t\psi+\zeta = \epsilon g\\
\end{cases}
\end{align*}
with $$f=\frac{1}{\mu\epsilon}(G[\epsilon\zeta]\psi-\Go\psi)$$ and $$g = - (\frac{1}{2}\vert\nablag\psi\vert^2-\frac{1}{\mu}\frac{(G[\epsilon\zeta]\psi+\epsilon\mu\nablag\zeta\cdot\nablag\psi)^2}{2(1+\epsilon^2\mu\vert\nablag\zeta\vert^2}).$$
\begin{remark}
As suggested by the notations, the quantity $f$ is of size $O(1)$  with respect to $\epsilon$ (although it has a $\frac{1}{\epsilon}$ factor), as given by the asymptotic extension of $G[\epsilon\zeta,0]$ given later in Proposition \ref{c4:asymptotic_G}. We are going to treat the non-linear terms $f,g$ as perturbations of the linear equation as $\epsilon$ goes to zero. \vspace{\baselineskip}
\end{remark}

For this proof, we need to use a dispersive estimate for the linear Water-Waves equations with flat bottom, proved in \cite{mesognon2}:
\begin{theorem}
Let \begin{equation*}\omega : \left\{ \begin{aligned}\R &\longrightarrow\R \\
\xi &\longmapsto \sqrt{\frac{\vert\xi\vert\tanh(\sqrt{\mu}\vert\xi\vert)}{\sqrt{\mu}}}.
\end{aligned}\right. \end{equation*}
Then, there exists $C>0$ independent on $\mu$ such that, for all $\mu>0$:
$$\forall t>0,\qquad \forall \varphi\in\mathcal{S}(\R)\qquad \vert e^{it\omega(D)}\varphi\vert_\infty \leq C(\frac{1}{\mu^{1/4}}\frac{1}{(t/\sqrt{\mu})^{1/8}}+\frac{1}{(t/\sqrt{\mu})^{1/2}})(\vert\varphi\vert_{H^1}+\vert x\partial_x\varphi\vert_2).$$ \label{c4:dispersive_estimate}
\end{theorem}
\begin{remark} The dispersion in $\frac{1}{t^8}$ stated by this Theorem is absolutely not optimal. Actually, one should expect a decay of order $\frac{1}{t^{1/3}}$ in dimension $1$ (see for instance \cite{mesognon2} for variants of Theorem \ref{c4:dispersive_estimate}). However, the $\frac{1}{t^{1/8}}$ decay suffices to prove the main result of this section.
\end{remark}
In view of use of Theorem \ref{c4:dispersive_estimate}, we also need a local existence result in weighted Sobolev spaces, proved also in \cite{mesognon2}. For all $N\in\N$, we define $\E^N_x$  by $$\E^N_x =  \E^N(\zeta,\psi)+\sum_{1\leq \vert\alpha\vert \leq N-2} \vert x\zetaa\vert_2^2+\vert\B x\psia\vert_2^2.$$
\begin{theorem}\label{c4:wlocal}
 Let us consider the assumption of Theorem \ref{c4:uniform_result}, and then consider $(\zeta,\psi)$ the unique solution provided by the Theorem \ref{c4:uniform_result}, of the Water-Waves equation \eqref{c4:ww_equu}. If $(\zeta^0,\psi^0)\in \E^N_x$, then one has $$(\zeta,\psi)\in L^\infty([0;T],\mathcal{E}^N_x),$$ with $$(\zeta,\psi)_{L^\infty([0;T],\mathcal{E}^N_x)} \leq C_2.$$
\end{theorem} 
We now state the main result of this section:
\begin{theorem}\label{c4:lintofull}
Let $(\zeta_0,\psi_0)$ be such that $(\zeta_0,\psi_0)\in\E^N_x$ .  Let denote $T>0$, $(\zeta^{WW},\psi^{WW})$ the solution of the Water-Waves equations  \eqref{c4:ww_equu} in $(\zeta,\psi)\in L^\infty([0;T],\mathcal{E}^N_x)$  given by Theorem \ref{c4:uniform_result} on $[0;T]$ and $(\zeta^L,\psi^L)$ the solution of the linearized equation \eqref{c4:zetapsi_eq} given by Proposition \ref{c4:existence_linear}, both with initial condition $(\zeta_0,\psi_0)$. Then, one has:
\begin{equation}\vert (\zeta^L,\B\psi^L)-(\zeta^{WW},\B\psi^{WW})\vert_{L^\infty((0;T);L^2(\R))} \leq ( \frac{\epsilon^{1/8}}{\mu^{3/16}}+\epsilon^{1/2}\mu^{1/4})C_2\label{c4:decroissance_l2}\end{equation} where $C_2$ is given by Theorem \ref{c4:uniform_result}.
\end{theorem}

\begin{proof}In all this proof, we will denote by $C^N$ any constant of the form \begin{equation}C^N=C(\E^N_x(\zeta^0,\psi^0),\frac{1}{h_{\min}},\frac{1}{a_0})\label{c4:defCN}\end{equation} where $C$ is a non decreasing function of its arguments.
Let us define $(\zeta,\psi) = (\zeta^{WW},\B\psi^{WW})-(\zeta^L,\B\psi^L)$ which is defined on $[0;T]$. We use the evolution operator $e^{\frac{t}{\epsilon}L}$ defined by \eqref{c4:def_L} to write:

\begin{equation*}
\partial_t (e^{\frac{t}{\epsilon}L}\begin{pmatrix}
\zeta\\
\psi
\end{pmatrix}) = e^{\frac{t}{\epsilon}L} F(\begin{pmatrix}
\zeta \\
\psi
\end{pmatrix})(t)
\end{equation*} where 

\begin{equation}
F(\begin{pmatrix}
\zeta \\
\psi
\end{pmatrix}) = \begin{pmatrix}
\frac{1}{\mu\epsilon}(G(\epsilon\zeta)\psi-\Go\psi) \\
-\frac{1}{2}\vert \nablag\psi\vert_2^2-\frac{1}{\mu}\frac{(G(\epsilon\zeta)\psi+\epsilon\mu\nablag\zeta\cdot\nablag\psi)^2}{2(1+\epsilon^2\mu\vert\nablag\zeta\vert^2)}	
\end{pmatrix}\label{c4:defF}
\end{equation}
and thus one has

\begin{equation}
\forall t\in[0;T],\qquad \begin{pmatrix}
\zeta \\
\psi
\end{pmatrix} = \int_0^t e^{\frac{s-t}{\epsilon}L}F(\begin{pmatrix}
\zeta \\
\psi
\end{pmatrix})(s)ds.\label{c4:integrale_dispersive}
\end{equation}
We set  $A=\begin{pmatrix} 
I &0\\
0 & \B
\end{pmatrix}$ and look for a estimate of the $L^2$ norm of $A~^t(\zeta,\psi)$. The proof consists in using the decay estimate of the linear operator $e^{tL}$ of the Water-Waves equation of Theorem \ref{c4:dispersive_estimate}. More precisely, the operator $F$ is "almost" bilinear (at least up to a $O(\epsilon)$ order term), which allows us to control the $L^2$ norm of the integral \eqref{c4:integrale_dispersive}  by writing estimates of the form

\begin{align*}
\vert \int_0^t Ae^{\frac{s-t}{\epsilon}L} F(\begin{pmatrix}
\zeta \\
\psi
\end{pmatrix})(s)ds\vert_2 &\leq \vert \int_0^t Ae^{\frac{s-t}{\epsilon}L}B(\begin{pmatrix}
\zeta \\
\psi
\end{pmatrix},\begin{pmatrix}
\zeta \\
\psi
\end{pmatrix})(s)ds\vert_2+\epsilon C
\end{align*}
where $C$ does not depend on $\epsilon$, and with $B$ a bilinear operator. The proof then consists in using a Strichartz type of estimate, using the dispersive nature of $e^{itL}$, given by  Theorem \ref{c4:dispersive_estimate} to get a control of the remaining integral of the form $$\vert\int_0^t e^{\frac{s-t}{\epsilon}L}B(\begin{pmatrix}
\zeta \\
\psi
\end{pmatrix},\begin{pmatrix}
\zeta \\
\psi
\end{pmatrix})(s)ds\vert_2\leq (\epsilon)^{1/8} t^{\alpha}$$ with $\alpha>0$. \par\vspace{\baselineskip}

We start to write $F$ under the form $F=B+\epsilon R$ where $B$ is bilinear, and $R$ is a least of size $O(1)$ with respect to $\epsilon$.  To this purpose, recalling that $F$ is given by \eqref{c4:defF}, we get inspired by the following Proposition, which gives an asymptotic extension of $\G[\epsilon\zeta,0]$ (see \cite{david} Proposition 3.44) with respect to $\epsilon$:

\begin{proposition}
Let $t_0>d/2,s\geq 0$ and $k=0,1$. Let $\zeta\in H^{s+(k+1)/2}\cap H^{t_0+2}(\R^d)$ be such that:
$$\exists h_{\min}>0,\forall X\in\R^d,1+\epsilon\zeta(X)\geq h_{\min}$$ and $\psi\in H^{s+k/2}(\R^d)$. We get:
$$\vert \G\psi-\Go\psi-\epsilon \mathcal{G}_1\vert_{H^{s-1/2}}  \leq \epsilon^2 \mu^{\frac{3+k}{4}} C(\frac{1}{h_{\min}},\mu_{\max},\vert\zeta\vert_{H^{t_0+1}},\vert \zeta\vert_{H^{s+(k+1)/2}}) \vert\B\psi\vert_{H^{s+k/2}},$$  where $\mathcal{G}_1 = -\Go(\zeta(\Go\cdot))-\mu\nablag\cdot(\zeta\nablag\cdot)$.\label{c4:asymptotic_G}
\end{proposition}

We therefore write \begin{align}
\frac{1}{\mu\epsilon}(G(\epsilon\zeta)-\Go)\psi = \frac{1}{\mu}\G_1-\frac{1}{\mu\epsilon}(\G_1+\Go-G(\epsilon\zeta))\psi\label{c4:F_decomp1}
\end{align}
where the second term of the right hand side satisfies the following estimate, using Proposition \ref{c4:asymptotic_G} with $k=1$ and $s=1/2$:
\begin{equation}
\vert \frac{1}{\mu\epsilon}(\G_1+\Go-G(\epsilon\zeta)) \vert_2\leq C(\frac{1}{h_{\min}},\mu_{\max},\vert\zeta\vert_{H^{t_0+1}},\vert \zeta\vert_{H^{1/2+1}}) \vert \B\psi\vert_{H^1}.	\label{c4:F_control1}
\end{equation}
The second component of $F$ (recall that it is given by \eqref{c4:defF}) is easier to decompose, and using Proposition \ref{c4:314} one gets 
\begin{equation}
-\frac{1}{2}\vert\nablag\psi\vert^2- \frac{1}{\mu}\frac{(G(\epsilon\zeta)\psi+\epsilon\mu\nablag\zeta\cdot\nablag\psi)^2}{2(1+\epsilon^2\mu\vert\nablag\zeta\vert^2)}	=-\frac{1}{2}\vert\nablag\psi\vert^2 - \frac{1}{\mu}(G(\epsilon\zeta)\psi)^2+\epsilon R(\psi)\label{c4:F_decomp2}
\end{equation}
with \begin{equation}\vert R(\psi)\vert_L^2 \leq C^N\label{c4:F_control2},\end{equation} where $C^N$ is given by \eqref{c4:defCN}. Using \eqref{c4:F_decomp1} with the control \eqref{c4:F_control1}, and \eqref{c4:F_decomp2} with the control \eqref{c4:F_control2}, one gets:
\begin{equation}F( \begin{pmatrix}
\zeta \\
\psi
\end{pmatrix}) = B(\begin{pmatrix}
\zeta \\
\psi
\end{pmatrix}),\begin{pmatrix}
\zeta \\
\psi
\end{pmatrix}))+\epsilon R\begin{pmatrix}
\zeta \\
\psi
\end{pmatrix})\label{c4:F_decomp}\end{equation}
with 
\begin{equation} B(\begin{pmatrix}
\zeta \\
\psi
\end{pmatrix}),\begin{pmatrix}
\zeta \\
\psi
\end{pmatrix}))=\begin{pmatrix}
\frac{1}{\mu}\G_1\psi \\
-\frac{1}{2}\vert\nablag\psi\vert^2-\frac{1}{\mu}G(\epsilon\zeta)\psi)^2
\end{pmatrix}
\label{c4:F_deomp}
\end{equation}
and \begin{equation}\vert  R\begin{pmatrix}
\zeta \\
\psi
\end{pmatrix})\vert_2\leq C^N.\label{c4:F_reste_control}\end{equation}
We therefore get, using \eqref{c4:F_decomp} and \eqref{c4:F_reste_control} (recall that $C^N$ is a constant of the form \eqref{c4:defCN}): \begin{align*}
\vert \int_0^t A e^{\frac{s-t}{\epsilon}L}F(\begin{pmatrix}
\zeta \\
\psi
\end{pmatrix})(s)ds\vert_2 &\leq \vert  \int_0^t e^{\frac{s-t}{\epsilon}L} A B(s)ds\vert_2+\epsilon C^N.
\end{align*}
We denote in the following lines $(\cdot,\cdot)_{L^2_x}$ the $L^2$ scalar product with respect to the space variable. One writes:
\begin{align*}
\vert  \int_0^t e^{\frac{s-t}{\epsilon}L} A B(s)ds\vert_2^2 &= \vert  \int_0^t\int_0^t (e^{\frac{s-t}{\epsilon}L} A B(s),e^{\frac{u-t}{\epsilon}L}AB(u))_{L^2_x}duds\vert.
\end{align*}
Now, note that $Ae^{itL}$ is sum of terms of the form $$\begin{pmatrix}
1 &i\omega(D)\\
+\frac{\B}{i\omega(D)} &\B
\end{pmatrix}e^{it\omega(D)}$$ and therefore, we are led to estimate terms of the form \begin{equation} \int_0^t\int_0^t (e^{\frac{s-t}{\epsilon}L}  f_2(s),e^{\frac{u-t}{\epsilon}L} f_1(u))_{L^2_x} duds,\label{c4:termsoftheform}\end{equation} where $f_1,f_2$ are terms of the form
\begin{equation}\frac{1}{\mu}\G_1,\qquad \omega(D)\frac{1}{\mu}G(\epsilon\zeta)^2,\qquad \omega(D)\vert\nablag\psi\vert^2,\qquad \frac{\B}{\omega(D)}\frac{1}{\mu}\G_1,\qquad \B\frac{1}{\mu}G(\epsilon\zeta)^2,\qquad \B \vert\nablag\psi\vert_2^2.\label{c4:forme_fi}\end{equation}One has to notice that the control related to the latter three terms can be deduced from the control of the first three terms since $\B$ acts like the square root $\omega(D)$ of $\Go$. Indeed, one has \begin{align*}\omega(D) &= \frac{\D}{(1+\sqrt{\mu}\D)^{1/2}}\sqrt{\frac{\tanh(\sqrt{\mu}\D)(1+\sqrt{\mu}\D)}{\sqrt{\mu}\D}}\\
&= \B \sqrt{\frac{\tanh(\sqrt{\mu}\D)(1+\sqrt{\mu}\D)}{\sqrt{\mu}\D}}
\end{align*}
and there exists $C_1$, $C_2>0$ independent of $\mu,\epsilon$ such that:
\begin{equation}\forall \xi\in\R^d,\qquad C_1\leq \sqrt{\frac{\tanh(\sqrt{\mu}\xig)(1+\sqrt{\mu}\xig)}{\sqrt{\mu}\xig}} \leq C_2.\label{c4:equivalence_B_omega}\end{equation}
We now estimate terms of the form \eqref{c4:termsoftheform} using a similar technique as for Strichartz estimates for dispersives PDE's. One computes, using the symmetry of $e^{itL}$:
\begin{align*}
 \int_0^t\int_0^t (e^{\frac{s-t}{\epsilon}L}  f_2(s),e^{\frac{u-t}{\epsilon}L} f_1(u))_{L^2_x} duds = \int_0^t\int_0^t   (f_2(s),e^{\frac{u+s-2t}{\epsilon}L} f_1(u))_{L^2_x} duds. \end{align*}
Now recall that $f_i$ are of the form \eqref{c4:forme_fi}. We are not treating all the possible cases, but only the most difficult one (the others are treated by the same technique), which is $f_1=f_2=\frac{1}{\mu}\G_1\psi$. Using the definition of $\G_1$ given in Proposition \ref{c4:asymptotic_G}, one computes, integrating by parts:
\begin{align}
&\int_0^t\int_0^t   (f_2(s),e^{\frac{u+s-2t}{\epsilon}L} f_1(u))_{L^2_x} duds\nonumber \\&= \int_0^t\int_0^t   (-\frac{1}{\mu}\Go(\zeta\Go(\psi))(s)-\nablag\cdot(\zeta\nablag\psi)(s),e^{\frac{u+s-2t}{\epsilon}L} \frac{1}{\mu}\G_1\psi(u))_{L^2_x} duds\nonumber\\
&= \int_0^t\int_0^t   (-\frac{1}{\mu}\zeta\Go(\psi)(s),e^{\frac{u+s-2t}{\epsilon}L} \Go\frac{1}{\mu}\G_1(u))_{L^2_x}+(\zeta\nablag\psi(s),e^{\frac{u+s-2t}{\epsilon}L}\nablag \frac{1}{\mu}\G_1\psi(u))_{L^2_x} duds.
\label{c4:termtotreat}\end{align}
We only control the first integral of the right hande side of \eqref{c4:termtotreat} (the other one is estimated by a similar technique). We set \begin{equation}\tilde{f}_2 =-\frac{1}{\mu}\zeta\Go\psi(s),\qquad \tilde{f}_1 = \Go \frac{1}{\mu}\G_1\psi(u)
\label{c4:def_ftilde}
\end{equation}
and we now use the dispersive estimate of Theorem \ref{c4:dispersive_estimate}:
\begin{align*}
\vert \int_0^t\int_0^t   (\tilde{f}_2(s),e^{\frac{u+s-2t}{\epsilon}L}( \tilde{f}_1)_{L^2_x} duds\vert &\leq \int_0^t\int_0^t  \vert\tilde{f_2}\vert_{L^1_x}\vert e^{\frac{u+s-2t}{\epsilon}L}( \tilde{f}_1)\vert_{L^\infty_x}\\
 &\leq C^N\int_0^t\int_0^t \frac{\epsilon^{1/8}}{\mu^{3/16}}\frac{1}{(2t-u-s)^{1/8}}\\&+\frac{\epsilon^{1/2}\mu^{1/4}}{(2t-u-s)^{1/2}}(\vert \tilde{f}_1(u)\vert_{H^1_x}+\vert x\partial_x \tilde{f}_1(u)\vert_{L^2_x})duds \\
&\leq( \frac{\epsilon^{1/8}}{\mu^{3/16}}t^{7/8}+\epsilon^{1/2}\mu^{1/4}t^{3/2}) C^N
\end{align*}
if one can prove the following controls:
\begin{equation}
\underset{u\in[0;T]}{\sup} \vert\tilde{f}_2(u)\vert_{ L^1_x} + \vert \tilde{f}_1\vert_{H^1_x}+\vert x\partial_x \tilde{f}_1(u)\vert_{L^2_x} \leq C^N,\qquad i=1,2.
\label{c4:controls_to_prove}
\end{equation}
One has, using Cauchy-Schwarz inequality:
\begin{align*}
 \vert\tilde{f}_2(u)\vert_{ L^1_x}  &\leq \vert\zeta\vert_2\vert\frac{1}{\mu}\Go\psi\vert_2\\
 &\leq \vert\zeta\vert_2\vert\B\psi\vert_{H^1}\\
 &\leq C^N
\end{align*}
where we used Proposition \ref{c4:314} with $\zeta=b=0$ to control $\frac{1}{\mu}\Go\psi$. We now focus on the most difficult term  of \eqref{c4:controls_to_prove} which is $\vert x\partial_x \Go \frac{1}{\mu}\G_1\psi\vert_{L^2_x}$. We use again the definition of $\G_1$ given by Proposition \ref{c4:asymptotic_G}, and we control $\frac{1}{\mu}\vert x\partial_x \Go \Go(\zeta\Go(\psi))\vert_2$ (the other one is similar). One computes:
\begin{align}
\frac{1}{\mu}\vert x\partial_x \Go \Go(\zeta\Go(\psi))\vert_2 &\leq \frac{1}{\mu}\vert x\Go^2(\partial_x\zeta)\Go\psi\vert_2+\frac{1}{\mu}\vert x\Go^2(\zeta\Go(\partial_x\psi))\vert_2. \label{c4:calcul1}
\end{align}
We only control the first term of the right hand side of \eqref{c4:calcul1} (the other one is similar):
\begin{align*}
\vert x\Go^2(\partial_x\zeta)\Go\psi\vert_2 &\leq\vert \Go^2 x(\partial_x\zeta) \Go\psi\vert_2+\vert [x,\Go^2](\partial_x\zeta)\Go\psi\vert_2.
\end{align*}
Now, note that for all $f\in\mathcal{S}(\R)$, one has:
\begin{align*}
\vert [x,\Go^2]f\vert_2 &= \vert \frac{d}{d\xi} (\sqrt{\mu}\vert\xi\vert\tanh(\sqrt{\mu}	\vert\xi\vert))^2\widehat{f}\vert_2 \\
&\leq \mu^2\vert f\vert_{H^3}
\end{align*}
using the definition of $\Go$ given by \eqref{c4:defgo}. Therefore, one has:
\begin{align*}
\frac{1}{\mu}\vert x\Go^2(\partial_x\zeta)\Go\psi\vert_2 &\leq  \mu\vert x(\partial_x\zeta) \Go\psi\vert_{H^4}+ \mu\vert (\partial_x\zeta)\Go\psi\vert_{H^3}\\
&\leq \mu^2 \vert x(\partial_x\zeta) \vert_{H^4}\vert \Go\psi\vert_{H^4},
\end{align*}
where we used the gross estimate $\vert\G_0^2 u\vert \leq \vert u\vert_{H^4}$. One can conclude using Proposition \ref{c4:314} with $\zeta=b = 0$,  and Theorem \ref{c4:wlocal}.

This achieves the proof of the controls \eqref{c4:controls_to_prove}. \par\vspace{\baselineskip}

\textbf{Conclusion:} We proved 	$$\vert (\zeta,\B\psi)(t)\vert_2 \leq (\frac{\epsilon^{1/8}}{\mu^{3/16}}t^{7/8}+\epsilon^{1/2}\mu^{1/4}t^{3/2})C^N.$$
One can then take the supremum over $t\in[0;T]$ and get the desired result.\qquad$\Box$
\end{proof}
\begin{remark}It is very important to note that according to Theorem \ref{c4:lintofull} the linear effects are a good approximation in the rigid lid regime if the ratio $\frac{\epsilon}{\mu^{3/2}}$ is not too large. For instance, if $\epsilon = \mu^{3/2}$, the strong convergence of the solutions of the fully nonlinear system to the solution of the linear system is not true anymore. In this case, the dispersive effects would be at the same order as the non-linear effects, and one should expect a behavior of solutions similar to one of the Korteweg de Vries's equation (such as solitary waves type behavior, see for instance \cite{benjamin1972stability}). In particular, it is no longer possible to treat the non-linearities $f,g$ as perturbations of the linearized system. A more precise study should be lead, and one should consider as approximation of the full nonlinear Water-Waves equation an equation where the nonlinear terms of size $\epsilon$ are taken into account (see for instance \cite{sautxu}). If $\epsilon \ll \mu^{3/2}$, then the non-linear effects are dominants on the dispersive effects, and one should expect breaking waves. This is the main interest of having a dispersive result of the form of Theorem \ref{c4:dispersive_estimate} which depends on the small parameters. \end{remark}\par\vspace{\baselineskip}

It is now easy to prove the following Corollary:

\begin{corollary}
Let $(\zeta_0,\psi_0)$ be such that $(\zeta_0,\psi_0)\in\E^N_x$ .  Let denote $T>0$, $(\zeta,\psi)$ the solution of the Water-Waves equations  defined on $[0;T]$ of \eqref{c4:ww_equu}. Then, $(\zeta,\B\psi)$ converge weakly to zero in $L^\infty_t L^2_x$ and does not converge strongly in this space.
\end{corollary}
\begin{proof}
Let us denote $(\zeta^L,\psi^L)$ the solution of the linearized equation \eqref{c4:zetapsi_eq} with initial condition $(\zeta_0,\psi_0)$ given by Theorem \ref{c4:wlocal} on $[0;T]$ (just take $\epsilon=0$ in the equation \eqref{c4:ww_equu} to get the solutions). Using Proposition \ref{c4:existence_linear}, one has $$(\zeta^L,\psi^L)(t) = e^{-\frac{t}{\epsilon}L}(\zeta^0,\psi^0).$$ Using the definition of $e^{tL}$ given by \eqref{c4:def_L}, $\zeta,\B\psi,$ are sum of terms of the form $e^{\frac{it\omega(D)}{\epsilon}}(\zeta^0,\psi^0)$ and therefore one has, using the dispersive estimate of Theorem \ref{c4:dispersive_estimate}:
$$\forall t\in [0;T],\qquad \vert(\zeta^L(t),\B\psi^L(t))\vert \leq (\frac{\epsilon^{1/8}}{\mu^{3/16}}t^{7/8}+\epsilon^{1/2}\mu^{1/4}t^{3/2})(\vert(\zeta^0,\B\psi^0)\vert_{H^1}+\vert x\partial_x\zeta^0,\ \B x\partial_x\psi^0)\vert_{L^2}$$ and therefore one gets the weak convergence of $(\zeta,\B\psi)$ in $L^\infty((0;T);L^2_x(\R))$ as $\epsilon$ goes to zero. The convergence is not strong in $L^\infty((0;T);L^2_x(\R))$ since the quantity $$\frac{1}{2\mu}(\Go\psi,\psi)_2+\frac{1}{2}\vert\zeta\vert_2^2\sim \frac{1}{2}(\vert\B\psi\vert_2^2+\vert\zeta\vert_2^2)$$ is conserved through time. \par\vspace{\baselineskip}

Now, according to Theorem \ref{c4:lintofull}, $(\zeta,\B\psi)-(\zeta^L,\psi^L)$ converges strongly to zero as $\epsilon$ goes to zero in $L^\infty((0;T);L^2_x(\R))$. Therefore, $(\zeta^L,\psi^L)$ converges weakly, but not strongly as $\epsilon$ goes to zero, in this space.$\qquad\Box$
\end{proof}

\section{Equivalence between the free surface Euler and Water-Wave equation}\label{c4:equivalence}

In all this Section, we do not make any assumption on the dimension, and we set the bottom to zero: $$d=1,2,\qquad b=0.$$  We consider the standard dimensionless version of the Water-Waves equations with flat bottom:
\begin{align}\begin{cases}
\ds\partial_t\zeta-\frac{1}{\mu}G[\epsilon\zeta,0]\psi = 0\\
\ds\partial_t\psi+\zeta+\frac{\epsilon}{2}\vert\nablag\psi\vert^2-\frac{\epsilon}{\mu}\frac{(G[\epsilon\zeta,0]\psi+\epsilon\mu\nablag\zeta\cdot\nablag\psi)^2}{2(1+\epsilon^2\mu\vert\nablag\zeta\vert^2}=0.\label{c4:wwequation}
\end{cases}
\end{align}

In Section \ref{c4:rigidlidlimitww}, we studied the rigid lid limit for the Water-Waves equations \eqref{c4:water_waves}. In order to complete the study of the rigid lid equations \eqref{c4:rigidlid1}, we study in this Section the rigid lid limit for the Euler equations \eqref{c4:eulerystem}. To this purpose, we prove that one can build rigorously solutions to the free surface Euler equations from the solutions of the Water-Waves equations. Such problem has been studied for instance in \cite{alazard}. The first problem is to define rigorously a solution to the free surface Euler equations. Because the free surface Euler equations \eqref{c4:eulerystem} are posed on a domain $\Omega_t^\epsilon$ which depends on $\epsilon$ and $t$, we have to be careful with the functional spaces used. The idea is that, by hypothesis, the height $1+\epsilon\zeta(t,X)$ of the water is bounded with respect to $X\in\mathbb{R}^d$ at a fixed time $t$. By continuity in time, we can include $\Omega_t^{\epsilon}$ in a fixed strip $\mathcal{S}^*$, which does not depends either on $t$ nor $\epsilon$.\par

We can then define rigorously what is a solution to the Euler equation \eqref{c4:eulerystem}: it is $(U,P)$ living in a functional space like "continuous in time with value to a Sobolev-type space in $X$ on the fixed strip $\mathcal{S}^*$" such that there exists $T>0$ such that, for all $0\leq t\leq T$, the derivative (in the sense of the distributions of $\mathcal{D}_{X,z,t}'([0,T]\times\mathcal{S}^*)$)  satisfies the equalities of (\ref{c4:eulerystem}). \par

\subsection{Main result}

In order to extract a convergent sub-sequence of solutions to the Water-Waves equations \eqref{c4:wwequation} as $\epsilon$ goes to zero, we need some compactness results. Indeed, a space such as $C([0;T];H^N(\R^d))$ is not compactly embedded into $L^\infty([0;T];H^N(\R^d))$. We need a control for time derivatives of the unknowns. For this reason, the stated local existence Theorem \ref{c4:uniform_result} is not sufficient. We recall here a slightly different version, including time derivatives in the energy space. The following framework is used for instance to prove the local existence for the Water-Waves equations with surface tension in  \cite{david} Chapter 9, and in \cite{benoit} for a long time existence result. We first introduce the following energy:

\begin{definition}\label{c4:defpsiat} Let $N\geq 1$ and $t_0>d/2$. We define \begin{equation}\E^N_1 = \vert\B\psi\vert_{H^{t_0+3/2}}^2+\sum_{\alpha\in\N^{d+1},\vert\alpha\vert \leq N} \vert\zetaa\vert_2^2+\vert\B\psia\vert_2^2\label{c4:energyt}\end{equation} with $$\zetaa = \partial^\alpha\zeta,\qquad \psia = \partial^\alpha\psi-\epsilon\underline{w}\partial^\alpha\zeta$$ and $$\underline{w} = \frac{G\psi+\epsilon\mu\nablag\zeta\cdot\nablag\psi}{1+\epsilon^2\mu\vert\nablag\zeta\vert^2}$$ with the notation:
$$\forall\alpha = (\alpha_1,...,\alpha_d,k)\in\N^{d+1}, \partial^\alpha = \partial_{X_1}^{\alpha_1}...\partial_{X_2}^{\alpha_d}\partial_t^k.$$
\end{definition}

\begin{remark}\begin{itemize}[label=--,itemsep=0pt]
\item It is very important to notice that we consider here also time derivatives in the energy. 
\item We recall that $\underline{w}$ coincides as suggested with the horizontal component of the velocity evaluated at the surface.\end{itemize}
\end{remark}

We now give a slightly different version of the local existence Theorem for the Water-Waves equations \eqref{c4:wwequation} proved by Alvarez-Samaniego and Lannes (see also \cite{david} Chapter 4 and Chapter 9):

\begin{theorem}[Alvarez Samaniego,Lannes]\label{c4:lannes}
Let $t_0>d/2$, $N>t_0+t_0\vee 2+3/2$, and $(\epsilon,\mu,\gamma)$ be such that $$0\leq \epsilon,\mu,\gamma \leq 1.$$ Let $(\zeta_0,\psi_0)$ with $\mathcal{E}_1^N(\zeta_0,\psi_0)<\infty$, where $\E_1^N$ is defined by \eqref{c4:energyt}. We assume that:
$$\exists h_{\min}>0,\exists a_0>0,\qquad 1+\epsilon\zeta_0\geq h_{\min} \text{ and }  \rt(\zeta_0,\psi_0)\geq a_0.$$ Then, there exists $T>0$ and a unique solution $(\zeta,\psi)\in C([0;T];H^{t_0+2}\times \dot{H}^2(\R^d))$ to the Water-Waves equation \eqref{c4:wwequation} such that $\E^N_1(\zeta,\psi) \in L^\infty([0;\frac{T}{\epsilon}])$. Moreover, one has $$\frac{1}{T} = C_1 \quad \text{ and }\quad \underset{t\in[0;\frac{T}{\epsilon}]}{\sup} \E^N_1(t)=C_2$$ with $C_i$ a constant of the form $C_i=C(\mathcal{E}^N(\zeta_0,\psi_0),\frac{1}{h_{\min}},\frac{1}{a_0})$ for $i=1,2$.
\end{theorem}

We use from now on the notations $H^{s,k}$ for the Beppo-Levy spaces, and $\Sigma_t^\epsilon$ for the diffeomorphism from the flat strip to the water domain (see Section \ref{c4:notations} for notations). We also recall that  $\Omega_t^\epsilon$ denote the domain occupied by the water at time $t$. We now state the main result of this Section:

\begin{theorem}\label{c4:equivasol}
Under the hypothesis of Theorem \ref{c4:lannes}, let $T>0$ be given by Theorem \ref{c4:lannes}. Then, for all $\epsilon >0$, there exists a flat strip $\mathcal{S}^*=\R^d\times (-(k+1);k)$ with $k\geq 1$ and $(U^\epsilon,\zeta^\epsilon,P^\epsilon)$ where $U^\epsilon = (V^\epsilon,w^\epsilon)$ in the following spaces: 
\begin{align}
\begin{cases}
\forall k\leq N,\quad\partial_t^k U^{\epsilon}\in L^{\infty}((0;\frac{T}{\epsilon});H^{N-k-1/2,N-k-1}(\mathcal{S}^*)^{d+1}) \\
\forall k\leq N,\quad \partial_t^k\zeta^{\epsilon}\in L^{\infty}([0;\frac{T}{\epsilon}];H^{N-k}(\mathbb{R}^d)) \\
\forall k\leq N, \quad \partial_t^k P^{\epsilon}\in L^{\infty}((0;\frac{T}{\epsilon});H^{N-k-1/2,N-k-1}(\mathcal{S}^*))
\end{cases}
\end{align}
with bounds uniform in $\epsilon$ in these spaces, such that for all $t\in [0;\frac{T}{\epsilon}]$, for all $\epsilon>0$, the domain $\Omega_t^\epsilon$ is included in $\mathcal{S}^*$. Moreover, $(V^\epsilon,w^\epsilon,\zeta^\epsilon,P^\epsilon)$  satisfy in the sense of the distributions of $\mathcal{D}_{X,z,t}'([0,\frac{T}{\epsilon}]\times\mathcal{S}^*)$ the Euler equations \eqref{c4:eulerystem}.
\end{theorem}

The proof follows these steps: \begin{itemize}[label=--,itemsep=0pt]
\item We claim some regularity for the velocity potential $\phi$ solving \eqref{c4:dirichletneumann} on the fixed strip $\mathcal{S}=\R^d\times (-1;0)$.
\item We extend $\phi$ to a larger strip than $\mathcal{S}$.
\item We recover regularity for the velocity potential $\Phi$ defined originally on the water domain by $\phi=\Phi\circ\Sigma_t^\epsilon$, after defining it on a large strip containing all the fluid domain  $\Omega_t^\epsilon$ for all $t,\epsilon$.\end{itemize}

\subsection{Regularity of the solutions of the Water Waves problem}
According to Theorem \ref{c4:lannes}, for all $\epsilon$, there exists a unique solution $(\zeta^{\epsilon},\psi^{\epsilon})$ to \eqref{c4:wwequation} on a time interval $[0,\frac{T}{\epsilon}]$, with $T$ only depending on the initial energy (and not on $\epsilon$), with the following control: 
\begin{equation}
\forall t\in [0;\frac{T}{\epsilon}], \qquad \mathcal{E}_1^N(U^{\epsilon})(t) \leq C_2\label{c4:t0_estimate}
\end{equation}
with $C_2$ a constant of the form $C(\mathcal{E}_1^N(\zeta_0,\psi_0),\frac{1}{h_{\min}},\frac{1}{a_0})$, where $\E_1^N$ is defined by \eqref{c4:energyt}. We can then prove the following regularity result:
\begin{proposition}
Let $(\zeta^{\epsilon},\psi^{\epsilon})$ be the solution of \ref{c4:wwequation} given by Theorem \ref{c4:lannes}, and $T>0$ such that (\ref{c4:t0_estimate}) is satisfied. Then, the following regularity results stand:
\begin{enumerate}[label=\arabic*),itemsep=0pt]
\item For all $0\leq k\leq N$, one has: $$ \partial_t^k\zeta^{\epsilon}\in L^{\infty}([0;\frac{T}{\epsilon}];H^{N-k}(\mathbb{R}^d))\quad\text{ and }\quad \vert \partial_t^k \zeta^{\epsilon}\vert_{L^{\infty}([0;\frac{T}{\epsilon}];H^{N-k}(\mathbb{R}^d))} \leq C_2;$$
\item For all $0\leq k \leq N-1$, one has: $$\partial_t^k\mathfrak{P}\psi^{\epsilon} \in L^{\infty}([0;\frac{T}{\epsilon}];H^{N-k-1/2}(\mathbb{R}^d))\quad\text{ and }\quad \vert \partial_t^k \mathfrak{P}\psi^{\epsilon}\vert_{L^{\infty}([0;\frac{T}{\epsilon}];H^{N-k-1/2}(\mathbb{R}^d))} \leq C_2,$$
\end{enumerate}
with $C_2$ a constant of the form $C(\mathcal{E}^N(\zeta_0,\psi_0),\frac{1}{h_{\min}},\frac{1}{a_0})$. \label{c4:propregularity}
\label{c4:controlt}\end{proposition}
\begin{proof}

Recall that the energy is: 
\begin{equation*}\mathcal{E}^N(U) = \modd{\mathfrak{P}\psi}_{H^{t_0+3/2}}+\underset{\alpha\in\mathbb{N}^{d+1},\mid\alpha\mid\leq N}{\Sigma} \modd{\zetaa}_2+\modd{\mathfrak{P} \psia}_2.\end{equation*}

1) Since the time derivatives of the unknowns appear in the energy, the control given by \eqref{c4:t0_estimate} implies that for all $0\leq k\leq N$, $\partial^k_t \zeta^{\epsilon}$ is in $L_t^\infty(H^{N-k}(\mathbb{R}^d))$, with the desired estimate. \par\vspace{\baselineskip}

2) Let us fix $k$ such that $0\leq k\leq N-1$. We have to get a control of $\mathfrak{P}\dt^k\partial^{\alpha}\psi$ by $\mathfrak{P}\psi_{(\alpha,k)}$. Adapting a proof from Lemma 4.6 in \cite{david}, one computes:
\begin{align*}
\vert\mathfrak{P}\partial_t^k\psi\vert_{H^{N-k-1/2}} &\leq \sum_{\beta\in\mathbb{N}^{d},\vert\beta\vert\leq N-1-k} \vert\mathfrak{P}\partial_t^k\partial^{\beta}\psi\vert_{H^{1/2}} \\
&\leq  \sum_{\beta\in\mathbb{N}^{d},\vert\beta\vert\leq N-1-k} (\vert\mathfrak{P}\psi_{(\beta,k)}\vert_{H^{1/2}}+\epsilon\vert\mathfrak{P}(\underline{w}\partial_t^k\partial^\beta\zeta)\vert_{H^{1/2}}) \end{align*} where we used the definition of $\psi_{(\beta,k)}$ given by Definition \ref{c4:defpsiat}. We now use the identity  $\vert\mathfrak{P} f\vert_{H^{1/2}} \leq \max\lbrace 1,\mu^{-1/4}\rbrace\vert\nabla f\vert_2$ to conclude:
\begin{align*}
\vert\mathfrak{P}\partial_t^k\psi\vert_{H^{N-k-1/2}} \leq  \sum_{\beta\in\mathbb{N}^{d},\vert\beta\vert\leq N-1-k} (\vert\mathfrak{P}\psi_{(\beta,k)}\vert_{H^{1}}+\epsilon \max\lbrace1,\mu^{-1/4}\rbrace\vert\underline{w}\vert_{H^{t_0+1}}\vert\partial_t^k\zeta\vert_{H^{N-k}}).
\end{align*}
One already has by the first point that $\vert\partial_t^k\zeta\vert_{H^{N-k}} \leq C_2$. Using Proposition \ref{c4:314}, we have $\vert\underline{w}\vert_{H^{t_0+1}}\leq \mu^{3/4} M \vert\B\psi\vert_{H^{t_0+3/2}}\leq C_2$. We now control  $\psi_{(\beta,k)}$ in $H^1$ norm, for all $\vert\beta\vert\leq N-1-k$. One computes:
\begin{align*}
\vert\B\psi_{(\beta,k)}\vert_{H^1} &\leq \vert\B\psi_{(\beta,k)}\vert_2+\sum_{\delta\in\N^d, \vert\delta\vert=1} \vert\partial^\beta\B\psi_{(\beta,k)}\vert_2\\
&\leq \vert\B\psi_{(\beta,k)}\vert_2 + \sum_{\delta\in\N^d, \vert\delta\vert=1} (\vert\B\psi_{(\beta+\delta,k)}\vert_2+\epsilon\vert \B(\partial^\delta\underline{w}\dt^k\partial^\beta\zeta)\vert_2).
\end{align*}
Since $\vert \beta+\delta\vert+k \leq N$, the first component of the right hand side from above is controlled by a constant of the form $C_2$. For the second component, we use the identity: $$\exists C>0,\qquad \forall f\in\mathcal{S}(\R^d),\qquad \vert\B f\vert\leq C \mu^{-1/4}\vert f\vert_{H^{1/2}}$$ where $C$ does not depend on $\mu$. Therefore, $$\vert \B(\partial^\delta\underline{w}\partial^\beta\zeta)\vert_2 \leq C\mu^{-1/4} \vert \partial^\delta\underline{w}\partial^\beta\zeta\vert_{H^{1/2}} \leq C\mu^{-1/4}\vert\underline{w}\vert_{H^{t_0+1}}\vert \dt^k\zeta\vert_{H^{N-k}}. $$ We now use as before Proposition \ref{c4:314} to control $\mu^{-1/4}\vert\underline{w}\vert_{H^{t_0+1}}$ by $C_2$ and we remark that $\vert\B\psi_{(\beta+\delta,k)}\vert_2\leq C_2$ for all $\vert\beta+\delta\vert+k\leq N$. We obtained the control: $$\vert\mathfrak{P}\partial_t^k\psi\vert_{H^{N-k-1/2}} \leq C_2$$ with $C_2$ independent on the time. Taking the supremum on time over $[0;\frac{T}{\epsilon}]$, one gets the desired result.

 $\qquad \Box $ 
\end{proof}

In order to define a velocity $U=\nablamug\Phi$ on the fluid domain $\Omega_t^\epsilon$, where $\Phi$ solves the Dirichlet-Neumann problem \eqref{c4:dirichletneumannnondim} on the fluid domain, we first study the regularity of $\phi$ where $\phi$ is the solution of the Dirichlet-Neumann problem \eqref{c4:dirichletneumann} on the flat strip $\S = \R^d\times (-1;0)$. We recall the relation $\phi = \Phi\circ \Sigma_t^\epsilon$, where $\Sigma_t^\epsilon$ is the diffeomorphism defined by \eqref{c4:diffeoS}.

\begin{proposition}
Let $\phi$ be the unique variational solution to $\eqref{c4:dirichletneumann}$. We have, for all $0\leq k\leq N-1$: $$\partial_t^k\nabla^{\mu,\gamma}\phi\in L^{\infty}((0,\frac{T}{\epsilon});H^{N-k-1/2,N-k-1}(\mathcal{S}))$$ with $$\vert\partial_t^k\nabla^{\mu,\gamma}\phi\vert_{ L^{\infty}((0,\frac{T}{\epsilon});H^{N-k-1/2,N-k-1}(\mathcal{S}))}\leq C_2$$ where $C_2$ is a constant of the form $C(\mathcal{E}^N(\zeta_0,\psi_0),\frac{1}{h_{\min}},\frac{1}{a_0})$.
\end{proposition}
This notation is a bit heavy, but we have to deal with the fact that we only use integer derivatives in $z$. 
\begin{proof}

  Using the notations of the Proposition \ref{c4:shapeanalicity}, we recall that:
\begin{equation*}
\phi = \mathfrak{A}_{\psi}(\zeta,0).
\end{equation*}
By taking the derivative in time, $\partial_t^k \phi$ is sum of terms of the form:
$$  d^j\mathfrak{A}_{\partial_t^{k-j}\psi}(\zeta,0)(\partial_t^{l_1}\zeta,...,\partial_t^{l_j}\zeta)$$
with $0\leq j\leq k$ and $l_1+...+l_j = j$. Let $l_{\max} = \underset{1\leq m\leq j}{\max} l_m$. We assume $l_{\max} = l_1$ by symmetry. We then have two cases: \par

1) Case $l_1 \leq N-t_0-1$. By using Proposition \ref{c4:a7}, with $t_0' = N-l_1-1$ (note that $t_0'\geq t_0 >d/2$) and $s\leq N-k-1/2$ we have:
\begin{align*}\vert \Lambda^s \nabla^{\mu,\gamma} d^j\mathfrak{A}_{\partial_t^{k-j}\psi}(\zeta,0)(\partial_t^{l_1}\zeta,...,\partial_t^{l_j}\zeta)\vert_2 &\leq \sqrt{\mu}M_0\prod_{m=1}^j \vert \epsilon \partial_t^{l_m}\zeta \vert_{H^{N-l_1}}\vert \mathfrak{P}\partial_t^{k-j}\psi\vert_{H^s}\\
&\leq C_2
\end{align*}
where we used the $H^{N-k-1/2}$ regularity of $\mathfrak{P}\partial_t^{k-j}\psi$ given by Proposition \ref{c4:propregularity}.\par\vspace{\baselineskip}

2) Case $l_1 \geq N-t_0-1$. Let us use Proposition \ref{c4:a8}, with $t_0'=N-1-(k-l_1)$ (just note that $t_0' \geq 3/2 >d/2$ for $ d=1,2$). We have, for all $s\leq N-k-1/2$ (we have $N-k-1/2 \geq t_0'$):
\begin{align*}\vert\Lambda^s\nabla^{\mu,\gamma} d^j\mathfrak{A}_{\partial_t^{k-j}\psi}(\zeta,0)(\partial^{l_1}\zeta,...,\partial^{l_j}\zeta)\vert_2 &\leq \sqrt{\mu}M_0\vert\epsilon\partial_t^{l_1}\zeta\vert_{H^{s+1/2}}\prod_{m>1} \vert\epsilon\partial_t^{l_m}\zeta\vert_{H^{t_0'+1}}\vert\mathfrak{P}\partial_t^{k-j}\psi\vert_{H^{t_0'+1/2}} \\ &\leq C_2
\end{align*}
because $t_0'+1/2+k-j\leq N-1/2$, and where we used again the Proposition \ref{c4:propregularity} for the  $H^{N-k-1/2}$ regularity of $\mathfrak{P}\partial_t^{k-j}\psi$.\par\vspace{\baselineskip}

We proved that $\partial_t^k\nabla^{\mu,\gamma}\phi\in L^{\infty}((0,\frac{T}{\epsilon};H^{N-k-1/2})$. Using the last part of Proposition \ref{c4:a7} and Proposition \ref{c4:a8}, we can adapt the previous proof to get $$\partial_t^k\nabla^{\mu,\gamma}\phi\in L^{\infty}((0,\frac{T}{\epsilon});H^{N-k-1/2,N-k-1}(\mathcal{S}))$$ with the desired control for the norm.$\qquad \Box $ 
\end{proof}

\subsection{Extension of the solution $\phi$}
We want to extend the distribution $\phi$ defined on the strip $\S=\R^d\times(-1;0)$ to a larger strip. This is the point of the following result:

\begin{theorem}
Let $s\in\R^+$ and $k\in\N$. Let us denote $\mathcal{S}_j = (-(j+1),j)\times\mathbb{R}^d$ for all $j\in\mathbb{N}$. Then, there exists an extension 

\begin{displaymath}
P :
\left.
  \begin{array}{rcl}
    H^{s,k}(\mathcal{S}_0) &\rightarrow &H^{s,k}(\mathcal{S}_j)  \\
    u&\mapsto &Pu \\
  \end{array}
\right.
\end{displaymath}
such that: 
\begin{equation}
\forall u\in H^{s,k}(S_0), \qquad\vert u\vert_{H^{s,k}(S_j)} \leq  C(k,j)\vert u\vert_{H^{s,k}(S_0)} \label{c4:prolongementth}
\end{equation}
where $C(k,j)$ only depends on $k$ and $j$. \label{c4:extension}
\end{theorem}
The proof is postponed to Appendix \ref{c4:appendixB}. We now have the full artillery to extend the solution $\Phi$ defined in the moving domain $\Omega_t^{\epsilon}$ into a function (with the same regularity) defined on a fixed strip $\mathcal{S}^*$. 

Let $\Sigma_t^\epsilon$ be the following diffeomorphism, mapping the flat strip $\mathcal{S} = \mathbb{R}^d\times (-1,0)$ into $\Omega_t^{\epsilon}$:
\begin{align*}
\Sigma_t^{\epsilon} : \mathbb{R}^{d+1} &\rightarrow \mathbb{R}^{d+1} \\
(X,z)&\mapsto (X,(1+\epsilon\zeta (t,X))z+\epsilon\zeta(t,X)).
\end{align*}

It is easy to check that $\Sigma_t^{\epsilon}$ is a homeomorphism that maps $\mathcal{S}$ exactly into $\Omega_t^{\epsilon}$. We now include $\Omega_t^{\epsilon}$ into a fixed strip $\mathcal{S}^*$:

\begin{proposition}
With the notations of Theorem \ref{c4:extension}, there exists a strip $\mathcal{S}_k$, which does not depend either in $\epsilon$ or $t$ such that:
$$\forall 0\leq t\leq \frac{T}{\epsilon},\qquad \Omega_t^{\epsilon} \subset S_k.$$
\end{proposition}

\begin{proof}
Proposition \ref{c4:propregularity} gives that:
\begin{align*}[0;\frac{T}{\epsilon}]\times \mathbb{R}^d &\rightarrow \mathbb{R} \\
(t,X)&\mapsto 1+\epsilon\zeta(t,X) 
\end{align*}
is (in particular) in $C([0;\frac{T}{\epsilon}],H^2(\mathbb{R}^d))$ with a bound uniform in $\epsilon$. The continuous embedding $H^2(\mathbb{R}^d) \subset L^{\infty}(\mathbb{R}^d)$ for $d\leq 2$ gives the desired result.$\qquad \Box $
\end{proof}

Now, with the uniform (with respect to $t,X,\epsilon$) bound  of $\zeta(t,X)$, it is easy to check that all strip $S_j$ is mapped by $\Sigma_t^{\epsilon}$ into a strip $\mathcal{S}_k$:
$$\forall j\in\mathbb{N}^*\quad \exists k\in\mathbb{N}^*, \quad\forall t\in [0,\frac{T}{\epsilon}],\quad\forall \epsilon \leq 1, \Sigma_t^{\epsilon} (\mathcal{S}_j) \subset \mathcal{S}_k.$$

\begin{definition}\label{c4:strips}
Let $k,l,j$ be such that: \begin{align*} \mathcal{S}_l &\subset \mathcal{S}_j \\
\Omega_t^{\epsilon}&\subset \mathcal{S}_k \text{ for all } t,\epsilon \\
 \mathcal{S}_l &\subset \Sigma_t^{\epsilon} (\mathcal{S}_j)  \text{ for all } t,\epsilon \\
\mathcal{S}_k &\subset \Sigma_t^{\epsilon} (\mathcal{S}_l) \text{ for all } t,\epsilon
\end{align*} We denote $\mathcal{S}_k = \mathcal{S}^*$, $\tilde{\phi}$ the extension of $\phi$ provided by theorem \ref{c4:extension} into $\mathcal{S}_l$, $\Phi = \phi \circ ({\Sigma_t^{\epsilon}}^{-1})$ and $\tilde{\Phi} = \tilde{\phi}\circ({\Sigma_t^{\epsilon}}^{-1})$.
\end{definition}

%
%
%
%
%
%
%
%
%

$\Phi$ is the potential on the domain $\Omega_t^{\epsilon}$, while $\tilde{\Phi}$ is an extension into the fixed strip $\mathcal{S}^*$. 
\begin{remark} The purpose of this definition is to have a distribution $\tilde{\Phi}$ extending $\Phi$ and defined on a flat strip $\S_k$. We therefore extend $\phi$ to $\S_l$, and $\tilde{\Phi}=\tilde{\phi} \circ ({\Sigma_t^{\epsilon}}^{-1})$ is then defined on $\S_k$ (since it contains $\Sigma_t^\epsilon(\S_l)$) which also contains $\Omega_t^\epsilon$. The fact that we need to use three strips instead of two is purely technical (see later the proof of Proposition \ref{c4:phiregularity}).
\end{remark}
\subsection{Regularity after diffeomorphism}
We now need to recover regularity for $\Phi$.
\begin{proposition}
Using the notations of Definition \ref{c4:strips}, let $\phi$ be a distribution on $[0;\frac{T}{\epsilon}]\times\mathcal{S}$  with the following regularity: $$\forall 0\leq k\leq N-1,\quad \partial_t^k\nabla^{\mu,\gamma}\phi\in L^\infty((0;\frac{T}{\epsilon});H^{N-k-1/2,N-k-1}(\mathcal{S}))$$ with a $C_2$ bound. Let $\tilde{\phi}$ be an extension of $\phi$ on the strip $\mathcal{S}_l$, and $\tilde{\Phi}=\tilde{\phi}\circ{\Sigma_t^\epsilon}^{-1}$.
Then, we have, for all $k\leq N-1 $: $$\partial_t^k\nabla^{\mu,\gamma} \tilde{\Phi}\in L^{\infty}((0;\frac{T}{\epsilon});H^{N-k-1/2,N-k-1}(\mathcal{S}^*)$$ with a $C_2$ bound.
\label{c4:phiregularity}\end{proposition}

The proof is postponed in Section \ref{c4:appendixB} for the sake of clarity.

\subsection{Craig-Sulem-Zakharov formulation to Euler formulation} 
We now build solutions to the free surface Euler equation \eqref{c4:eulerystem} from the potential $\tilde{\Phi}$ introduced before. We first prove that $\tilde{\Phi}$ satisfies the Bernoulli equations on the domain $\Omega_t^\epsilon$. One considers the following distribution of $\mathcal{D}'((0;\frac{T}{\epsilon});\mathcal{S}^*)$:
$$\partial_t\tilde{\Phi} + \frac{\epsilon}{2}\vert\nabla^{\mu,\gamma}\tilde{\Phi}\vert^2+\frac{1}{\epsilon}z,$$
which exists in $L^{\infty}((0;\frac{T}{\epsilon});H^{N-1/2,N-1}(\mathcal{S}^*))$. \par

Let us consider a measurable function $F$ of the variable $t$, such that for almost every $t\in (0;\frac{T}{\epsilon})$, F is equal to this quantity. For almost every $t$, $F$ is an element of $H^{N-1/2,N-1}(\mathcal{S}^*))$ which trace on the boundary $z=\epsilon \zeta$ of $\Omega_t^{\epsilon}$ is equal to zero (this statement implies the second equation of the Craig-Sulem-Zakharov formulation \eqref{c4:wwequation}). For a complete and rigorous proof of this statement, see \cite{alazard}. In particular, for almost every $t$, $F$ is equal to an element of \begin{equation*}H_{0,surf}^{N-1/2,N-1}(\mathcal{S}^*) = \overline{ \mathcal{S}^* \setminus \lbrace z=\epsilon\zeta\rbrace}_{\vert\vert.\vert\vert_{H^{N-1/2,N-1}(\mathcal{S}^*)}}.\end{equation*}

We denote this element $-\frac{1}{\epsilon}P$. Of course, $P$ corresponds physically to the pressure. We can rewrite this as the Bernoulli equation: 

\begin{equation}\partial_t\tilde{\Phi} + \frac{\epsilon}{2}\vert\nabla^{\mu,\gamma}\tilde{\Phi}\vert^2+\frac{1}{\epsilon}z = -\frac{1}{\epsilon}P\label{c4:bernouilli}\end{equation}
in $\mathcal{D}'((0;\frac{T}{\epsilon});\Omega_t^{\epsilon})$, where $((0;\frac{T}{\epsilon});\Omega_t^{\epsilon}) = \lbrace (t,X,z)\in (0;\frac{T}{\epsilon})\times\mathcal{S}^*, (X,z)\in\Omega_t^{\epsilon}\rbrace$ (note that this is an open set of $(0;\frac{T}{\epsilon})\times\mathcal{S}^*$). 
\begin{remark}
It is easy to deduce from the Bernoulli equation \eqref{c4:bernouilli} that for all $k\leq N_0$: 
$$\partial_t^k P\in L^{\infty}((0;\frac{T}{\epsilon});H^{N-k-1/2,N-k-1}(\mathcal{S}^*)$$ with a bound uniform in $\epsilon$. \label{c4:pregularity}
\end{remark}

Now, defining $U=\nabla^{\mu,\gamma}\tilde{\Phi}$, we have immediately 
\begin{align}\label{c4:eulerfree}\begin{cases}
\displaystyle{\partial_t V + \epsilon (V\cdot\nabla^{\gamma} + \frac{1}{\mu}w\partial_z)V = -\frac{1}{\epsilon} \nabla^{\gamma} P} \text{ in } \mathcal{D}'((0;\frac{T}{\epsilon});\Omega_t^{\epsilon})\\
\displaystyle{\partial_t w + \epsilon (V\cdot\nabla^{\gamma} + \frac{1}{\mu}w\partial_z)w = -\frac{1}{\epsilon} (\partial_z P+1)} \text{ in } \mathcal{D}'((0;\frac{T}{\epsilon});\Omega_t^{\epsilon})\\
\displaystyle{\partial_t \zeta +\epsilon\underline{V}\cdot\nabla^{\gamma}\zeta - \frac{1}{\mu}\underline{w} = 0 }  \text{ in } \mathcal{D}'((0;\frac{T}{\epsilon});\R^d) \\
\nablamug\cdot(U) = 0 \text{ in } \mathcal{D}'((0;\frac{T}{\epsilon});\Omega_t^{\epsilon})\\
\mbox{\rm curl}^{\mu,\gamma}(U) = 0 \text{ in } \mathcal{S}^*\\
U\cdot n = 0 \text { on } z=-1
\end{cases}\end{align}

\begin{remark}
- The third equation of the Euler system \eqref{c4:eulerfree} is just the first equation of the Zakharov-Craig-Sulem equation. \par 
- Note that the regularity of $U = \nabla^{\mu,\gamma}\tilde{\Phi}$ is given by Proposition \ref{c4:phiregularity}

\end{remark}

\section{The lack of strong convergence from Euler to rigid-lid}
We now have a proper set of solutions $(U^\epsilon,\zeta^\epsilon,P^\epsilon)$ to the free surface Euler equations \eqref{c4:eulerystem}, defined in the following spaces:
\begin{align}
\begin{cases}
\forall k\leq N-1,\quad\partial_t^k U^{\epsilon}\in L^{\infty}((0;\frac{T}{\epsilon});H^{N-k-1/2,N-k-1}(\mathcal{S}^*)) \\
\forall k\leq N-1,\quad \partial_t^k\zeta^{\epsilon}\in L^{\infty}([0;\frac{T}{\epsilon}];H^{N-k}(\mathbb{R}^d)) \\
\forall k\leq N-1, \quad \partial_t^k P^{\epsilon}\in L^{\infty}([0;\frac{T}{\epsilon}];H^{N-k-1/2,N-k-1}(\mathcal{S}^*))\label{c4:Ureuglarite}
\end{cases}
\end{align}
with bounds uniform in $\epsilon$ of the form $C(\mathcal{E}^N(\zeta_0,\psi_0),\frac{1}{h_{\min}},\frac{1}{a_0})$. In order to prove the weak convergence to the solutions of the rigid lid model, we need to prove some compactness results in the Beppo-Levy spaces in order to extract a convergent sub-sequence of solutions.
\subsection{Compactness result}
We recall that $A\subset\subset B$ means that there exists a compact $K$ such that $A\subset K\subset B$.
\begin{lemma}\label{c4:compactemb}
Let $k\geq 1$. We have the following compact embedding:
$$H^{s,k}_{loc}(\mathcal{S}^*) \subset H^{s-1,k-1}_{loc}(\mathcal{S}^*) \quad\forall 1\leq k\leq s$$\label{c4:lemmacompact}
\end{lemma}
\begin{proof}
The proof remains the same as for the Rellich-Kondrachov theorem, in the case of Sobolev spaces (see \cite{brezis}). Let $\omega\subset\subset \mathcal{S}^*$. \par 
1) We first show the compact embedding of $H^{s,1}(\omega)$ into $H^{s-1,0}(\omega)$.  We write, for $h=(h_1,h_2)\in \mathbb{R}\times\mathbb{R}^d$, $\vert h\vert < d(\omega,(\mathcal{S}^*)^\complement)$ and $u\in C^{\infty}(\mathcal{S}^*)$:
\begin{align*}
&\int_{\omega}\vert (\Lambda^{s-1}u)(X-h_1,z-h_2)-(\Lambda^{s-1}u)(X,z)\vert^2 dXdz \\ &\leq \int_{\omega} \vert \int_0^1 (\nabla_{X,z}\Lambda^s u)(X-th_1,z-th_2).(h_1,h_2)dt\vert^2dXdz.\end{align*} We now use Jensen's inequality to write:
\begin{align*}
&\int_{\omega}\vert (\Lambda^{s-1}u)(X-h_1,z-h_2)-(\Lambda^{s-1}u)(X,z)\vert^2 dXdz\\
&\leq  \int_{\mathcal{S}^*}  \int_0^1 \vert\nabla_{X,z} \Lambda^{s-1} u(X,z)\vert^2\vert h\vert^2dtdXdz\\
&\leq \vert u\vert_{H^{s,1}(\mathcal{S}^*)}^2\vert h\vert^2
\end{align*}
and the result stands true for all $u\in H^{s,1}(\mathcal{S}^*)$ by density of $C^{\infty}(\mathcal{S}^*)$ in this space. \par\vspace{\baselineskip}
It follows from the theorem of Riesz-Frechet-Kolmogorov that if $(u_n)_n$ is bounded in $H^{s,1}(\mathcal{S}^*)$, then ${({\Lambda^{s-1}u_n}_{\mid\omega})}_n$ is precompact in $L^2(\omega)$. Such result is true for all $\omega\subset\subset \mathcal{S}^*$ and therefore, by a diagonal sub-sequence argument, one proves the compact embedding $H^{s,1}_{loc}(\mathcal{S}^*)$ into $H^{s-1,0}_{loc}(\mathcal{S}^*)$.
\par

2) For the general case, let $k\in\mathbb{N}^*$ and $s\geq k$, let $(u_n)_n$ be bounded in $H^{s,k+1}(\mathcal{S}^*)$ with $s\geq k+1$. By the definition of Beppo-Levy spaces (recall Section \ref{c4:notations}), to prove the convergence of a sub-sequence in $H^{s-1,k}_{loc}(\mathcal{S}^*)$, one must prove the existence of a sub-sequence ${(u_{\varphi(n)})}_n$ such that for all $0\leq l\leq k$, the sequence ${(\partial_z^l u_{\varphi(n)})}_n$ converges in $L^2_{loc}((-(k+1),k), H^{s-1-l}_{loc}(\R^d))$ (recall that $\S^* = (-(k+1);k)$). From the case $k=1$, there exists $u\in L^2(\mathcal{S}^*)$ and a sub-sequence  ${(u_{\varphi(n)})}_n$ convergent to $u$ in $L^2_{loc}(\mathcal{S}^*)$. In particular, ${(\partial_z^l u_{\varphi(n)})}_n$ is convergent to $\partial_z^l u$ in $\mathcal{D}'(\mathcal{S}^*)$. But one has also ${(\partial_z^l u_{\varphi(n)})}_n \in H^{s-l,k+1-l}(\mathcal{S}^*)$ and therefore this sequence converges, up to a sub-sequence, in $L^2_{loc}((-(k+1),k), H^{s-1-l}_{loc}(\R^d))$, since $k+1-l\geq 1$. By uniqueness of the limit in $\mathcal{D}'(\mathcal{S}^*)$, the limit is $\partial_z^l u$.\qquad $\Box$
\end{proof}

\subsection{Extraction of a convergent sub-sequence for $U^\epsilon$}
We now perform the rigid lid time scaling: $$t'=\epsilon t$$ and the change of unknown $P'=\frac{1}{\epsilon^2}(P+z)$. We now have $$U^\epsilon\in L^\infty((0;T);H^{N-1/2,N-1}(\mathcal{S}^*)^{d+1}),\zeta^\epsilon\in L^\infty((0;T);H^N(\R^d))$$ and the following equations are satisfied:
\begin{align}\label{c4:eulerfreers}\begin{cases}
\displaystyle{\partial_t V^\epsilon +  (V^\epsilon\cdot\nabla^{\gamma} + \frac{1}{\mu}w^\epsilon\partial_z)V^\epsilon = -\frac{1}{\epsilon} \nabla^{\gamma} P^\epsilon} \text{ in } \mathcal{D}'((0;T);\Omega_t^{\epsilon})\\
\displaystyle{\partial_t w^\epsilon +  (V^\epsilon\cdot\nabla^{\gamma} + \frac{1}{\mu}w^\epsilon\partial_z)w^\epsilon = -\frac{1}{\epsilon} (\partial_z P^\epsilon)} \text{ in } \mathcal{D}'((0;T);\Omega_t^{\epsilon})\\
\displaystyle{\partial_t \zeta^\epsilon +\underline{V}^\epsilon\cdot\nabla^{\gamma}\zeta^\epsilon - \frac{1}{\mu\epsilon^2}\underline{w}^\epsilon = 0 } \text{ in } \mathcal{D}'((0;T);\R^d) \\
\nablamug\cdot(U^\epsilon) = 0 \text{ in } \mathcal{D}'((0;T);\Omega_t^{\epsilon})\\
\mbox{\rm curl}^{\mu,\gamma}(U^\epsilon) = 0 \text{ in } \mathcal{S}^*\\
U^\epsilon\cdot n = 0 \text { on } z=-1
\end{cases}\end{align}
We now prove the following Theorem:
\begin{theorem}
There exists $U,P,\zeta$ distributions of $D'((0;T);\Omega)$ such that $$U\in H^{N-3/2,N-2}(\S^*), U\in H^{N-5/2,N-3}(\S^*),\zeta\in H^{N-1}(\R^d)$$ are solutions in the distributional sense of $D'((0;T);\Omega)$ of the rigid lid equations  \eqref{c4:rigidlid}. Moreover, the solutions of \eqref{c4:eulerfreers} converge in the distributional sense of $D'((0;T);\Omega_t^\epsilon)$ to these solutions of the rigid-lid equation \ref{c4:rigidlid}. The convergence of $(U^\epsilon)_\epsilon$ is not strong in $L^\infty((0;T);L^2(\Omega))$.
\end{theorem}

By Lemma \ref{c4:compactemb}, there exists $U^\epsilon\in H^{N-3/2,N-2}(\S^*)$ such that $(U^\epsilon)_\epsilon$ converges up to a sub-sequence to $U$ in $H^{N-3/2,N-2}_{loc}(\S^*)$. Moreover, since $(U^\epsilon)_\epsilon \in L^\infty((0;T);H^{N-1/2,N-1}(\S^*)^{d+1})$, one has also $U \in L^\infty((0;T);H^{N-3/2,N-2}(\mathcal{S}^*)^{d+1})$. Therefore, $(U^\epsilon)_\epsilon$ converges up to a sub-sequence to $U$ in $D'((0;T);\S^*)$. Therefore, $(\partial_t U^\epsilon)_\epsilon$  converges (up to a sub-sequence) to $\partial_t U$ in $D'((0;T);\S^*)$. \par\vspace{\baselineskip}
In view of taking the limit as $\epsilon$ goes to zero in the first equation of \eqref{c4:eulerfreers}, one must prove the convergence of the non-linear terms. We do it for the term $V^\epsilon\cdot\nablag V^\epsilon$. One uses the following estimate, which proof can be found in \cite{david} Corollary B.5. 
$$\vert fg\vert_{H^{N-3/2,N-2}} \leq C \vert f\vert_{H^{N-1/2,N-1}}\vert g\vert_{H^{N-3/2,N-2}}$$ for all $N-2\geq d/2$ (which is true for our choice of $N$, since $N-2\geq t_0+3/2 >d/2$).  
Such estimate proves that $V^\epsilon\cdot\nablag V^\epsilon \in L^\infty H^{N-3/2,N-2}$, and therefore one has the convergence of this term (up to a sub-sequence) as $\epsilon$ goes to zero, in $D'((0;T);\S^*)$. It proves, using the equation \eqref{c4:eulerfreers} that $(\nabla^\gamma P^\epsilon)_\epsilon$ converges in $D'((0;T);\S^*)$ to an element $G$. By De Rham's Theorem, there exists $P\in D'((0;T);\S^*)$ such that $G=\nablag P$. The same study can be done for the second equation of \ref{c4:eulerfreers} and one obtains, passing to the limit as $\epsilon$ goes to zero:
\begin{align*}
\displaystyle{\partial_t V +  (V\cdot\nabla^{\gamma} + \frac{1}{\mu}w\partial_z)V = - \nabla^{\gamma} P} \text{ in } \mathcal{D}'((0;T);\Omega)\\
\displaystyle{\partial_t w +  (V\cdot\nabla^{\gamma} + \frac{1}{\mu}w\partial_z)w = -\frac{1}{\epsilon} (\partial_z P)} \text{ in } \mathcal{D}'((0;T);\Omega
\end{align*}
with $U=(V,w)\in L^\infty((0;T);H^{N-3/2,N-2}(\mathcal{S}^*)^{d+1}, \nablamug P\in L^\infty((0;T);H^{N-5/2,N-3}(\mathcal{S}^*)^{d+1}$. \par\vspace{\baselineskip}  We now pass to the limit in the third equation (on $\zeta^\epsilon$) of \eqref{c4:eulerfreers}. The same reasoning as before proves that $(\zeta^\epsilon)_\epsilon$ converges up to a sub-sequence to a $\zeta\in H^{N-1}(\R^d)$ in $H^{N-1}_{loc}(\R^d)$. Therefore, the same convergence occurs in $D'((0;T);\S^*)$, and $(\partial_t\zeta)_\epsilon $ converges up to a sub-sequence to $\partial_t \zeta$ in $D'((0;T);\S^*)$. We prove as before the convergence of the bi-linear term $(\underline{V}^\epsilon\cdot\nabla^{\gamma}\zeta^\epsilon)_\epsilon$ in $D'((0;T);\S^*)$. Multiplying the third equation of \ref{c4:eulerfreers} by $\epsilon^2$, one finds at the limit the following equation:
$$\underline{w}=0 \text { in } \mathcal{D}'((0;T);\Omega_t^{\epsilon}).$$
Finally, one obtains at the limit $\epsilon$ goes to zero, the following equations:
\begin{align*}\begin{cases}
\displaystyle{\partial_t V +  (V\cdot\nabla^{\gamma} + \frac{1}{\mu}w\partial_z)V =  -\nabla^{\gamma} P} \text{ in } \Omega \\
\displaystyle{\partial_t w +  (V\cdot\nabla^{\gamma} + \frac{1}{\mu}w\partial_z)w =  -(\partial_z P) }\text{ in } \Omega \\
\underline{w} = 0 \\
\nabla^{\mu,\gamma}\cdot U = 0 \text{ in } \Omega \\
\mbox{\rm curl}^{\mu,\gamma}(U) = 0 \text{ in } \Omega \\
U\cdot n = 0 \text{ for }  z=-1 \\
\end{cases}\end{align*}
and the study above prove the following regularity for the unknowns: $U\in H^{N-3/2,N-2}(\S^*),\nablamug P\in U\in H^{N-5/2,N-3}(\S^*),\zeta\in H^{N-1}(\R^d)$. The sequence $(U^\epsilon)_\epsilon$ does not converge in $L^\infty((0;T);L^2(\Omega)$. Indeed, it would imply the strong convergence of $\B\psi$ in $L^2(\R^d)$ which is not satisfied, according to Theorem \ref{c4:theorem_nonstrongcv}.   $\qquad\Box$

\begin{appendix}

\section{The Dirichlet Neumann Operator}\label{c4:appendixA}
Here are for the sake of convenience some technical results about the Dirichlet Neumann operator, and its estimates in Sobolev norms. See \cite{david} Chapter 3 for complete proofs. The first two  propositions give a control of the Dirichlet-Neumann operator. 

\begin{proposition}\label{c4:314}
Let $t_0$>d/2, $0\leq s \leq t_0+3/2$ and $(\zeta,\beta)\in H^{t_0+1}\cap H^{s+1/2}(\mathbb{R}^d)$ such that  \begin{equation*} \exists h_0>0,\forall X\in\mathbb{R}^d,  \epsilon\zeta(X)-\beta b(X) +1 \geq h_0\end{equation*} 

(1)\quad The operator $G$ maps continuously $\overset{.}H{}^{s+1/2}(\mathbb{R}^d)$ into $H{}^{s-1/2}(\mathbb{R}^d)$ and one has 
\begin{equation*}
\vert G\psi\vert_{H^{s-1/2}} \leq \mu^{3/4} M(s+1/2) \vert\mathfrak{P}\psi\vert_{H^s},
\end{equation*}
where $M(s+1/2)$ is a constant of the form $C(\frac{1}{h_0},\vert\zeta\vert_{H^{t_0+1}},\vert b\vert_{H^{t_0+1}},\vert\zeta\vert_{H^{s+1/2}},\vert b\vert_{H^{s+1/2}})$.
\par

(2)\quad The operator $G$ maps continuously $\overset{.}H{}^{s+1}(\mathbb{R}^d)$ into $H{}^{s-1/2}(\mathbb{R}^d)$ and one has 
\begin{equation*}
\vert G\psi\vert_{H^{s-1/2}} \leq \mu M(s+1) \vert\mathfrak{P}\psi\vert_{H^s+1/2},
\end{equation*}
where $M(s+1)$ is a constant of the form $C(\frac{1}{h_0},\vert\zeta\vert_{H^{t_0+1}},\vert b\vert_{H^{t_0+1}},\vert\zeta\vert_{H^{s+1}},\vert b\vert_{H^{s+1}})$.
\par

Moreover, it is possible to replace $G$ by $\underline{w}$ in the previous result, where $\underline{w} = \frac{G\psi+\epsilon\mu\nablag\zeta\cdot\nablag\psi}{1+\epsilon^2\mu\vert\nablag\zeta\vert^2}$(vertical component of the velocity $U=\nabla_{X,z}\Phi$ at the surface).
\end{proposition}

\begin{proposition}\label{c4:318}
Let $t_0>d/2$, and $0\leq s\leq t_0+1/2$. Let also $\zeta,b\in H^{t_0+1}(\mathbb{R}^d)$ be such that  $$\exists h_0>0, \forall X\in\mathbb{R}^d, 1+\epsilon\zeta(X)-\beta b(X) \geq h_0$$ Then, for all $\psi_1$, $\psi_2\in \overset{.}H{}^{s+1/2}(\mathbb{R}^d)$, we have  $$(\Lambda^sG\psi_1,\Lambda^s\psi_2)_2 \leq\mu M_0 \vert \mathfrak{P}\psi_1\vert_{H^s}\vert \mathfrak{P}\psi_2\vert_{H^s},$$
where $M_0$ is a constant of the form $C(\frac{1}{h_0},\vert\zeta\vert_{H^{t_0+1}},\vert b\vert_{H^{t_0+1}})$.
\end{proposition}

The second result gives a control of the shape derivatives of the Dirichlet-Neumann operator. More precisely, we define the  open set  $\mathbf{\Gamma}\subset H^{t_0+1}(\mathbb{R}^d)^2$ as:
$$\mathbf{\Gamma} =\lbrace \Gamma=(\zeta,b)\in H^{t_0+1}(\mathbb{R}^d)^2,\quad \exists h_0>0,\forall X\in\mathbb{R}^d, \epsilon\zeta(X) +1-\beta b(X) \geq h_0\rbrace$$ and, given a $\psi\in \overset{.}H{}^{s+1/2}(\mathbb{R}^d)$, the mapping: \begin{equation}\label{c4:mapping}G[\epsilon\cdot,\beta\cdot] : \left. \begin{array}{rcl}
&\mathbf{\Gamma} &\longrightarrow H^{s-1/2}(\mathbb{R}^d) \\
&\Gamma=(\zeta,b) &\longmapsto G[\epsilon\zeta,\beta b]\psi.
\end{array}\right.\end{equation} We can prove the differentiability of this mapping. The following Theorem gives a very important explicit formula for the first-order partial derivative of $G$ with respect to $\zeta$:

\begin{theorem}\label{c4:321}
Let $t_0>d/2$. Let $\Gamma = (\zeta,b)\in \mathbf{\Gamma}$ and $\psi\in\overset{.}H{}^{3/2}(\R^d)$. Then, for all $h\in H^{t_0+1}(\R^d)$, one has $$dG(h)\psi = -\epsilon G(h\underline{w})-\epsilon\mu\nablag\cdot(h\Vd),$$ with $$\underline{w}=\frac{G\psi+\epsilon\mu\nablag\zeta\cdot\nablag\psi}{1+\epsilon^2\mu\modd{\nablag\zeta}},\quad\text{ and }\quad \Vd = \nablag\psi-\epsilon\underline{w}\nablag\zeta.$$
\end{theorem}

The following result gives estimates of the derivatives of the mapping \eqref{c4:mapping}. 
\begin{proposition}\label{c4:328}
Let $t_0$>d/2, $0\leq s \leq t_0+1/2$ and $(\zeta,\beta)\in H^{t_0+1}(\mathbb{R}^d)$ such that: \begin{equation*} \exists h_0>0,\forall X\in\mathbb{R}^d,  \epsilon\zeta(X)-\beta b(X) +1 \geq h_0\end{equation*} 
Then, for all $\psi_1,\psi_2\in \overset{.}H{}^{s+1/2}(\mathbb{R}^d)$, for all $(h,k)\in H^{t_0+1}(\mathbb{R}^d)$ one has  
\begin{equation*}
\vert (\Lambda^s d^j G(h,k)\psi_1,\Lambda^s\psi_2)\vert \leq \mu M_0 \prod_{m=1}^j \vert(\epsilon h_m,\beta k_m)\vert_{H^{t_0+1}}\vert\mathfrak{P}\psi_1\vert_s\vert\mathfrak{P}\psi_2\vert_s,
\end{equation*}
where $M_0$ is a constant of the form $C(\frac{1}{h_0},\vert\zeta\vert_{H^{t_0+1}},\vert b\vert_{H^{t_0+1}})$.
\end{proposition}

The following Proposition gives the same type of estimate that the previous one:

\begin{proposition}\label{c4:328b}
Let $t_0>d/2$ and $(\zeta,b)\in H^{t_0+1}$ be such that \begin{equation*} \exists h_0>0,\forall X\in\mathbb{R}^d,  \epsilon\zeta(X)-\beta b(X) +1 \geq h_0.\end{equation*} Then, for all $0\leq s\leq t_0+1/2$, $$\vert d^j G(h,k)\psi\vert_{H^{s-1/2}} \leq M_0 \mu^{3/4} \prod_{m=1}^j \vert (\epsilon h_m,\beta k_m)\vert_{H^{t_0+1}} \vert \B\psi\vert_{H^s}$$
\end{proposition}
We  need the following commutator estimate:

\begin{proposition}\label{c4:329}
Let $t_0>d/2$ and $\zeta, b \in H^{t_0+2}(\mathbb{R}^d)$ such that:  \begin{equation*} \exists h_0>0,\forall X\in\mathbb{R}^d,  \epsilon\zeta(X)-\beta b(X) +1 \geq h_0\end{equation*} 
For all $\underline{V}\in H^{t_0+1}(\mathbb{R}^d)^2$ and $u\in H^{1/2}(\mathbb{R}^d)$, one has 

\begin{equation*}
((\underline{V}\cdot\nabla^{\gamma} u),\frac{1}{\mu}Gu)\leq M\vert\underline{V}\vert_{W^{1,\infty}}\vert \mathfrak{P} u\vert_2^2,
\end{equation*}
where $M$ is a constant of the form $C(\frac{1}{h_0},\vert\zeta\vert_{H^{t_0+2}},\vert b\vert_{H^{t_0+2}})$.
\end{proposition}
In order to have regularity estimates of the potential $\Phi$ (recall that it solves the Dirichlet-Neumann problem (\ref{c4:dirichlet})), we  need to introduce the following mapping:

\begin{definition}
Let  $t_0>d/2$. \par
(1) We denote by $\mathbf{\Gamma}\subset H^{t_0+1}(\mathbb{R}^d)^2$ the open set:
$$\mathbf{\Gamma} =\lbrace \Gamma=(\zeta,b)\in H^{t_0+1}(\mathbb{R}^d)^2,\quad \exists h_0>0,\forall X\in\mathbb{R}^d, \epsilon\zeta(X) +1-\beta b(X) \geq h_0\rbrace$$
(2) One can define, for all $0\leq s\leq t_0+1/2$ and $\psi\in \dot{H}^{s+1/2}(\mathbb{R}^d)$ a mapping $\mathfrak{A}_{\psi}$ as  
\begin{displaymath}
\mathfrak{A}_{\psi} :
\left.
  \begin{array}{rcl}
    \mathbf{\Gamma} & \rightarrow &\dot{H}^{s+1}\mathcal{S} \\
    \Gamma & \mapsto & \Phi_{\Gamma} \\
  \end{array}
\right.
\end{displaymath}
where $\Phi_{\Gamma}$ is the unique variational solution to (\ref{c4:dirichletneumann}).
\end{definition}

One can prove the analicity of $\mathfrak{A}_{\psi}$, which means that the solution $\phi_{\Gamma}$ of (\ref{c4:dirichletneumann}) is analytic with respect to the boundaries. The following result then gives some estimates about the derivatives of the mapping $\mathfrak{A}_{\psi}$
\begin{proposition}
Let $t_0>d/2$ and $0\leq s\leq t_0+1/2,\psi\in H^{s+1/2}(\mathbb{R}^d)$, and $\Gamma = (\zeta,b)\in\mathbf{\Gamma}$. For all $j\in\mathbb{N}$ and $(h,k)=(h_1,...,h_j,k_1,...,k_j)\in H^{t_0+1}(\mathbb{R}^d)^2$, one has 
$$ \vert\Lambda^s \nabla^{\mu,\gamma}d^j \mathfrak{A}_{\psi}(\Gamma)(h,k)\vert_2\leq\sqrt{\mu}M_0 \prod_{m=1}^j \vert(\epsilon h_m,\beta k_m)\vert_{H^{t_0+1}}\vert\mathfrak{P}\psi\vert_{H^s},$$ \label{c4:shapeanalicity}
where $M_0$ is a constant of the form $C(\frac{1}{h_0},\vert\zeta\vert_{H^{t_0+1}},\vert b\vert_{H^{t_0+1}})$.
\\If $s=t_0+1/2$, then the same estimate holds on $\vert\vert \nabla^{\mu,\gamma}d^j\mathfrak{A}_{\psi}(\Gamma)(h,k)\vert\vert_{H^{t_0+1/2,k}(\mathcal{S})}$ for $k\leq t_0+1/2$. \label{c4:a7}
\end{proposition}
In Proposition \ref{c4:a7} below, we need at least the $H^{t_0+1/2}$-norm of the component $(h,k)$, even if $s=0$. The following proposition allows to relax this constraint, by using only the $H^{s+1/2}$ norm of the first component of $(h,k)$.
\begin{proposition}\label{c4:a8}
Let $t_0>d/2$ and $0\leq s\leq t_0,\psi\in H^{t_0+1/2}(\mathbb{R}^d)$, and $\Gamma = (\zeta,b)\in\mathbf{\Gamma}$. For all $j\in\mathbb{N}$ and $(h,k)=(h_1,...,h_j,k_1,...,k_j)\in H^{t_0+1}(\mathbb{R}^d)^2$, one has 

$$ \vert\Lambda^s \nabla^{\mu,\gamma}d^j \mathfrak{A}_{\psi}(\Gamma)(h,k)\vert_2\leq\sqrt{\mu}M_0 \vert(\epsilon h_1,\beta k_1)\vert_{H^{s+1/2}}\prod_{m>1}^j \vert(\epsilon h_m,\beta k_m)\vert_{H^{t_0+1}}\vert\mathfrak{P}\psi\vert_{H^{t_0+1/2}},$$
where $M_0$ is a constant of the form $C(\frac{1}{h_0},\vert\zeta\vert_{H^{t_0+1}},\vert b\vert_{H^{t_0+1}})$.
\end{proposition}

\section{Extensions on Beppo-Levy Spaces}\label{c4:appendixB}
This section contains proofs of the existence and regularity of extensions used in Section \ref{c4:equivalence}.
\begin{theorem}
Let $s\in\R^+$ and $k\in\N$. Let us denote $\mathcal{S}_j = (-(j+1),j)\times\mathbb{R}^d$ for all $j\in\mathbb{N}$. Then, there exists an extension 

\begin{displaymath}
P :
\left.
  \begin{array}{rcl}
    H^{s,k}(\mathcal{S}_0) &\rightarrow &H^{s,k}(\mathcal{S}_j)  \\
    u&\mapsto &Pu \\
  \end{array}
\right.
\end{displaymath}
such that: 
\begin{equation}
\forall u\in H^{s,k}(S_0), \qquad\vert u\vert_{H^{s,k}(S_j)} \leq  C(k,j)\vert u\vert_{H^{s,k}(S_0)} \label{c4:prolongementth2}
\end{equation}
where $C(k,j)$ only depends on $k$ and $j$. 
\end{theorem}
\begin{proof}
We first construct an extension to $H^{s,k}((-1,1)\times\mathbb{R}^d)$.
\par

The proof requires a small adaptation from the case $k=1$ which is proved in \cite{brezis}. When $k=1$, $u$ is extended by reflection:
\begin{displaymath}
\tilde{u}(X,z) = 
\left\lbrace
  \begin{array}{rcl}
    &u(X,z) \quad &\forall z\in (-1,0)  \\
    &u(X,-z) \quad &\forall z\in (0,1).
  \end{array}
\right.
\end{displaymath}

But such extension $\tilde{u}$ is indeed in $H^{s,1}(S_1)$ but is not in general in $H^{s,2}(S_1)$ since derivatives in $z$ of $\tilde{u}$ differs for $z=0$: 
\begin{equation*}
\partial_z \tilde{u}(X,0^+) \neq -\partial_z \tilde{u}(X,0^-) 
\end{equation*}
We should therefore define an extension which derivatives of order $i\leq k-1$ have the same value in $0^+$ and in $0^-$. We are looking for an extension under the form:
\begin{displaymath}
Pu(X,z) = 
\left\lbrace
  \begin{array}{rcl}
    u(X,z) &\forall z\in (-1,0) \\
    \sum_{i=0}^{k-1} c_i u(X,-\alpha_i z) &\forall z\in (0,1).
  \end{array}
\right.
\end{displaymath}

The condition over the $\alpha_i$ and $c_i$, $0\leq i\leq k-1$ is:
\begin{equation*}
\forall 0\leq j\leq k-1,\qquad \sum_{i=0}^{k-1} c_i (-\alpha_i)^j \partial_z^j u(X,0^+) = \partial_z^j u(X,0^-)
\end{equation*} 

i.e. $\sum_{i=0}^k c_i (-\alpha_i)^j = 1$. We should therefore find $(c_0,...,c_{k-1})$ and $(\alpha_0,...,\alpha_{k-1})$ such that:
\begin{equation}
 \begin{pmatrix}
1&\cdots&1 \\
-\alpha_1&\cdots &-\alpha_{k-1} \\
(-\alpha_1)^2&\cdots&(-\alpha_{k-1})^2 \\
\vdots&\vdots&\vdots
\end{pmatrix}
 \begin{pmatrix}
c_0 \\
\vdots \\
\vdots \\
c_{k-1}
\end{pmatrix}
=
 \begin{pmatrix}
1 \\
\vdots \\
\vdots \\
1
\end{pmatrix}\label{c4:matrixextension}
\end{equation}
which is a Vandermonde system. It suffices to take the $\alpha_i$ distincts, non zero, and taken in $(0,1)$ in order for  $u(X,-\alpha_i z)$ to make sense. Then the $c_i$ are defined as solutions of the Vandermonde system \eqref{c4:matrixextension}.
\par\vspace{\baselineskip}

Now, let us prove that such defined extension $Pu$ maps continuously $H^{s,1}(S_1)$ into \newline $H^{s,1}((-1,1)\times\mathbb{R}^d)$. \par\vspace{\baselineskip}

1) Case $u\in C^{\infty}(\overline{\mathcal{S}_0})$
\par

It is clear that $Pu$ is measurable. Let us check that it is $L^2((-1,1),H^s(\mathbb{R}^d)$:
\begin{align*}\int_{-1}^1\int_{\mathbb{R}^d} \Lambda^{2s} \vert Pu(X,z)\vert^2dXdz &= \int_{-1}^0\int_{\mathbb{R}^d} \Lambda^{2s} \vert u(X,z)\vert^2dXdz + \int_{0}^1\int_{\mathbb{R}^d} \Lambda^{2s} \vert \sum_{i=0}^{k-1} c_i u(X,-\alpha_i z)\vert^2dXdz \\
&\leq \vert u\vert_{H^{s,k}((-1,0)\times\mathbb{R}^d)}^2 + \sum_{i=0}^{k-1}\frac{c_i^2}{\alpha_i}\int_{-\alpha_i}^0\int_{\mathbb{R}^d}\Lambda^{2s}\vert u(X,z)\vert^2dXdz \\
&\leq C_k \vert u\vert_{H^{s,k}((-1,0)\times\mathbb{R}^d)}^2 .
\end{align*}
The first inequality is Cauchy-Schwarz's inequality. The constant $C_k$ only depends on $k$. Now, let us check that $Pu$ is $H^{k-j}((-1,1);H^{s-k+j}(\mathbb{R}^d))$ for all $0\leq j \leq k$. Let $\varphi\in C_0^{\infty}((-1,1)\times\mathbb{R}^d)$. One set $j$ such that $0\leq j\leq k$, and computes $\partial_z^j (Pu)$ in the distributional sense of $D'((-1,1)\times\R^d)$:

\begin{align*}\int_{-1}^1\int_{\mathbb{R}^d}   u(X,z)\partial_z^j\varphi(X,z)dXdz &= \int_{-1}^0\int_{\mathbb{R}^d}   u(X,z)\partial_z^j\varphi(X,z)dXdz \\&+ \int_{0}^1\int_{\mathbb{R}^d}   \sum_{i=0}^{k-1} c_i u(X,-\alpha_i z)\partial_z^j\varphi(X,z)dXdz \\
&= \int_{-1}^0\int_{\mathbb{R}^d}   (-1)^j\partial_z^j u(X,z)\varphi(X,z)dXdz\\ &+ \sum_{l=0}^{j-1} (-1)^{j-1-l} \int_{\mathbb{R}^d} \partial_z^{j-l-1}u(X,0)\partial_z^{l}\varphi(X,0)dXdz \\
&+ \int_{0}^1\int_{\mathbb{R}^d}(-1)^j\sum_{i=0}^{k-1}c_i (-\alpha_i)^j(\partial_z^j u)(X,-\alpha_i z)\varphi(X,z)dXdz \\
&+ \sum_{l=0}^{j-1} (-1)^{j-l} \int_{\mathbb{R}^d} c_i(-\alpha_i)^{j-l-1}\partial_z^{j-l-1}u(X,0)\partial_z^l\varphi(X,0)dXdz
\end{align*}
by integrating by parts, (recall that $u\in C^{\infty}(\overline{\mathcal{S}_1})$). Note that these calculus does not make sense for $u\in H^{s,k}(\mathcal{S}_0)$ since $\varphi$ is not in $C_0^{\infty}(\mathcal{S}_0)$.
\par

Since we have the identity 

$$\forall 0\leq j\leq k-1, \sum_{i=0}^{k-1} c_i (-\alpha_i)^j \partial_z^j u(X,0) = \partial_z^j u(X,0)$$
the two integrals over $\mathbb{R}^d$ cancel one another. Therefore, we have: 
\begin{align}\int_{-1}^1\int_{\mathbb{R}^d}   u(X,z)\partial_z^j\varphi(X,z)dXdz &=\int_{-1}^0\int_{\mathbb{R}^d}   (-1)^j\partial_z^j u(X,z)\varphi(X,z)dXdz\nonumber \\&+ \int_{0}^1\int_{\mathbb{R}^d}\sum_{i=0}^{k-1}c_i (\alpha_i)^j(\partial_z^j u)(X,-\alpha_i z)\varphi(X,z)dXdz \label{c4:prolongement}\end{align}
which prove that $$\partial_z^j(Pu) = (1-sgn(z))\partial_z^j u(X,z) + (1+sgn(z))\sum_{i=0}^{k-1}c_i (-\alpha_i)^j(\partial_z^j u)(X,-\alpha_i z) $$
on $\mathcal{D}'((-1,1)\times\mathbb{R}^d)$, with the notation $sgn(z)=1$ if $z\geq 0$ and $sgn(z)=-1$ if $z<0$. It is then quite easy to check that $\partial_z^j(Pu) \in H^{k-j}((-1,1);H^{s-k+j}(\mathbb{R}^d)$ (proceed as for case $j=0$) with 
\begin{equation}\vert Pu\vert_{H^{k-j}((-1,1);H^{s-k+j}(\mathbb{R}^d)} \leq C_k  \vert u\vert_{H^{k-j}((-1,0);H^{s-k+j}(\mathbb{R}^d)}\label{c4:control_prolongement}\end{equation}
2) Case $u\in  H^{s,k}(\mathcal{S}_0)$
\par

By density of $C^{\infty}(\overline{\mathcal{S}_0})$ in $  H^{s,k}(\mathcal{S}_0)$, it is easy to check that (\ref{c4:prolongement}) stands true in $H^{s,k}(\mathcal{S}_0)$, and the result is proved, with the control (\ref{c4:control_prolongement}).
\\
We can construct by exactly the same way an extension of $u\in H^{s,k}(\mathcal{S}_0)$ into $H^{s,k}((-2,0)\times\mathbb{R}^d)$ and then combine the two extensions to have an extension to $H^{s,k}(\mathcal{S}_1)$. By the same way it is easy to construct the extension into  $H^{s,k}(\mathcal{S}_k)$.  Finally, the main  of the theorem is proved, with the control (\ref{c4:prolongementth2}). $\qquad \Box $ 
\end{proof}

\begin{proposition}
Using the notations of Definition \ref{c4:strips}, let $\phi$ be a distribution on $[0;T_0]\times\mathcal{S}$  with the following regularity: $$\forall 0\leq k\leq N,\quad \partial_t^k\nabla^{\mu,\gamma}\phi\in L^\infty((0;T_0);H^{N-k-1/2,N-k-1}(\mathcal{S}))$$ Let $\tilde{\phi}$ be an extension of $\phi$ on the strip $\mathcal{S}_l$, and $\tilde{\Phi}=\tilde{\phi}\circ{\Sigma_t^\epsilon}^{-1}$.
Then, we have, for all $k\leq N $: $$\partial_t^k\nabla^{\mu,\gamma} \tilde{\Phi}\in L^{\infty}((0;T_0);H^{N-k-1/2,N-k-1}(\mathcal{S}^*)$$ with a bound uniform in $t$ and $\epsilon$.
\end{proposition}
\begin{proof}
The proof consists in proving that if $\Theta$ is a function defined on $[0;T_0]\times\mathcal{S}^*$, with $\theta = \Theta\circ\Sigma_t^{\epsilon}$ such that: $$\forall k\leq N, \partial_t^k \theta \in L^{\infty}((0;T_0);H^{N-k+1/2,N-k}(\mathcal{S}^j))$$ then we have the same regularity for $\Theta$: $$\forall k\leq N, \partial_t^k \Theta \in L^{\infty}((0;T_0);H^{N-k+1/2,N-k}(\mathcal{S}^*))$$ with a control of the norm:
\begin{equation}\vert\partial_t^k\Theta\vert_{L^{\infty}((0;T_0);H^{N-k+1/2,N-k}(\mathcal{S}^*))}\leq C \vert\partial_t^k\theta\vert_{L^{\infty}((0;T_0);H^{N-k+1/2,N-k}(\mathcal{S}^j))}\label{c4:diffeonorm}\end{equation}
where $C$ does not depend on $\epsilon$.

Let us do it by induction on $k$. The induction hypothesis is the following: \par
"For all $N\in\mathbb{N}$ and $\theta\in L^{\infty}(0;T_0;H^{N+1/2,N}(\mathcal{S}^j))$ if $\Theta = \theta\circ{\Sigma_t^{\epsilon}}^{-1}$, then we have: $$ \partial_t^k\Theta\in L^{\infty}(0;T_0;H^{N-k+1/2,N-k}(\mathcal{S}^*))$$ with a control \begin{equation}\vert\partial_t^k\Theta\vert_{L^{\infty}((0;T_0);H^{N-k+1/2,N-k}(\mathcal{S}^*))}\leq C \vert\partial_t^k\theta\vert_{L^{\infty}((0;T_0);H^{N-k+1/2,N-k}(\mathcal{S}^j))}\label{c4:induction}\end{equation}
where $C$ does not depend on $\epsilon$"

1) k = 0\par
Let us prove that $\Theta\in L^{\infty}(0;T_0,H^{N+1/2,N}(\mathcal{S}^*))$. We are using the following characterization of $H^{N+1/2}(\mathbb{R}^d)$: \begin{equation*} \vert u\vert_{H^{N+1/2}(\mathbb{R}^d)}^2 \sim \vert u\vert_{H^{N}(\mathbb{R}^d)}^2 +\int_{\mathbb{R}^d}\int_{\mathbb{R}^d} \frac{\vert D^N u(x)-D^N u(y)\vert^2}{\vert x-y\vert^{2(1/2+d/2)}}dxdy\end{equation*} 
\par

Let first prove that $\Theta\in L^{\infty}((0;T_0);L^2(-(k+1);k),H^{N+1/2}(\mathbb{R}^d))$ with the control \eqref{c4:induction} by induction on $N$. Let the induction hypothesis be:\par
"For for all  $\theta\in L^{\infty}((0;T_0);L^2(-(l+1);l),H^{N+1/2}(\mathbb{R}^d))$,if $ \Theta = \theta\circ{\Sigma_t^{\epsilon}}^{-1}$ then we have:$$\Theta \in L^{\infty}((0;T_0);L^2(-(k+1);k),H^{N+1/2}(\mathbb{R}^d))$$ 
with a norm control such as \eqref{c4:induction}" \par
For $N=0$, we need the previous characterization of $H^{1/2}(\mathbb{R}^d)$:
\begin{align*}
&\int_{\mathbb{R}^d}\int_{\mathbb{R}^d}\int_{-(k+1)}^k \frac{\vert \Theta(X,z) - \Theta(Y,z)\vert^2}{\vert X-Y\vert^{2(1/2+d/2)}}dXdYdz \\
&= \int_{\mathbb{R}^d}\int_{\mathbb{R}^d}\int_{-(k+1)}^k \frac{\vert \theta\circ{\Sigma_t^{\epsilon}}^{-1}(X,z) - \theta\circ{\Sigma_t^{\epsilon}}^{-1}(Y,z)\vert^2}{\vert X-Y\vert^{2(1/2+d/2)}}dXdYdz \\
&=\int_{{\Sigma_t^{\epsilon}}^{-1}(-(k+1))_z}^{{\Sigma_t^{\epsilon}}^{-1}(k)_z}\int_{\mathbb{R}^d}\int_{\mathbb{R}^d} \frac{\vert \theta(X,u) - \theta\circ{\Sigma_t^{\epsilon}}^{-1}(Y,u)\vert^2}{\vert X-Y\vert^{2(1/2+d/2)}}\vert J_{\Sigma_t^{\epsilon}}(X)\vert dXdYdu \\
&\leq \int_{-(l+1)}^l \int_{\mathbb{R}^d}\int_{\mathbb{R}^d} \frac{\vert \theta(X,u) - \theta\circ{\Sigma_t^{\epsilon}}^{-1}(Y,u)\vert^2}{\vert X-Y\vert^{2(1/2+d/2)}}\vert J_{\Sigma_t^{\epsilon}}(X)\vert dXdYdu\\ 
&=  \int_{{\Sigma_t^{\epsilon}}^{-1}(-(l+1))_z}^{{\Sigma_t^{\epsilon}}^{-1}(l)_z} \int_{\mathbb{R}^d}\int_{\mathbb{R}^d} \frac{\vert \theta(X,v) - \theta(Y,v)\vert^2}{\vert X-Y\vert^{2(1/2+d/2)}}\vert J_{\Sigma_t^{\epsilon}}(X)\vert \vert J_{\Sigma_t^{\epsilon}}(Y)\vert dXdYdv \\
&\leq  \int_{-(j+1)}^j \int_{\mathbb{R}^d}\int_{\mathbb{R}^d} \frac{\vert \theta(X,v) - \theta(Y,v)\vert^2}{\vert X-Y\vert^{2(1/2+d/2)}}\vert J_{\Sigma_t^{\epsilon}}(X)\vert \vert J_{\Sigma_t^{\epsilon}}(Y)\vert dXdYdv \\
&\leq C\vert \theta\vert_{L^{\infty}((0;T_0);H^{1/2,0}(\mathcal{S}^j))}
\end{align*}
with $C$ uniform in $\epsilon$. We used the change of variable $(X,z) = \Sigma_t^{\epsilon}(X,u)$ and $(Y,u) = \Sigma_t^{\epsilon}(Y,v)$ in the integrals. The first inequality comes from $\mathcal{S}^k\subset \Sigma_t^{\epsilon}(\mathcal{S}^l) $. The last comes from $\mathcal{S}^l\subset\Sigma_t^{\epsilon}(\mathcal{S}^j)$.\par
Suppose the result true for $N-1$ with $N\geq 1$. Let prove it for $N$. We have:
\begin{align*}\nabla_{X,z}\Theta &= \nabla_{X,z}(\theta(X,\frac{z-\epsilon\zeta}{h_B})) \\
&=(\nabla_{X,z}\theta)(X,\frac{z-\epsilon\zeta}{h_B})+\nabla_{X,z}(\frac{z-\epsilon\zeta}{h_B})(\partial_z\theta)(X,\frac{z-\epsilon\zeta}{h_B}) \end{align*}

We have $(\nabla_{X,z}\theta)\in  L^{\infty}((0;T_0);L^2(-(k+1);k),H^{N-1+1/2}(\mathbb{R}^d))$, so by induction hypothesis, this term is controlled. For the latter one, just remark that $\nabla_{X,z}(\frac{z-\epsilon\zeta}{h_B})\in  H^{t_0}(\mathbb{R}^d) $ , so with classical product estimates: 
\begin{align*}
&\int_{-(k+1)}^k \vert \nabla_{X,z}(\frac{z-\epsilon\zeta}{h_B})(\partial_z\theta)(X,\frac{z-\epsilon\zeta}{h_B})\vert_{H^{N-1+1/2}(\mathbb{R}^d)}^2 dz \\&\leq \int_{-(l+1)}^l \vert \nabla_{X,z}(\frac{z-\epsilon\zeta}{h_B})\vert_{H^{t_0}}^2\vert(\partial_z\theta)(X,\frac{z-\epsilon\zeta}{h_B})\vert_{H^{N-1+1/2}(\mathbb{R}^d)}^2 dz \\
&\leq C \int_{-(k+1)}^k \vert(\partial_z\theta)(X,\frac{z-\epsilon\zeta}{h_B})\vert_{H^{N-1+1/2}(\mathbb{R}^d)}^2 dz \\
&= C\vert (\partial_z\theta)\circ{{\Sigma_t^{\epsilon}}^{-1}}\vert_{L^2((-(k+1);k);H^{N-1+1/2}(\mathbb{R}^d))} \\
&\leq C \vert \partial_z\theta\vert_{L^2((-(l+1);l);H^{N-1+1/2}(\mathbb{R}^d))}  \\
&\leq C \vert \partial_z\theta\vert_{L^2((-(j+1);j);H^{N-1+1/2}(\mathbb{R}^d))} 
	\end{align*}
using the induction hypothesis and $\mathcal{S}_l\subset \mathcal{S}_j$. The constant $C$ does not depends either on $t$ or $\epsilon$, so the result is proved. \par


2) Now, suppose this result is proved for $k-1$, with $k\geq 1$. Let prove it for $k$. We write:
\begin{align*}\partial_t \Theta &= \partial_t\theta(t,X,\frac{z-\epsilon\zeta}{h_B})\\
&=(\partial_t\theta)(t,X,\frac{z-\epsilon\zeta}{h_B}) + \partial_t(\frac{z-\epsilon\zeta}{h_B})(\partial_z\theta)(t,X,\frac{z-\epsilon\zeta}{h_B})
\end{align*}
Since $\partial_t\theta$ is $L^{\infty}(0;t0);H^{N-1+1/2,N-1}(\mathbb{R}^d)$,  by induction hypothesis the $(k-1)-$th time derivative of the first term of the r.h.s. is controlled as in \eqref{c4:induction}.\par

 The same argument stands for 	
$(\partial_z\theta)(t,X,\frac{z-\epsilon\zeta}{h_B})$. The term $\partial_t(\frac{z-\epsilon\zeta}{h_B})$ is harmless, since it is in $L^{\infty}((0;T_0);H^{N-1+1/2,N-1})$ with $N-1\geq t_0$, so standard product estimates in Sobolev spaces will finally give the control of the second term of the $(k-1)-$th time derivative of the r.h.s.
$\qquad \Box $ 
\end{proof}

\end{appendix}

The author has been partially funded by the ANR project Dyficolti ANR-13-BS01-0003-01.

\bibliographystyle{plain}

\bibliography{defaut_convergence}

\end{document}